\documentclass[a4paper,12pt]{amsart}

\usepackage[utf8]{inputenc}
\usepackage[american]{babel}
\usepackage{amsmath,amssymb,amsthm}
\usepackage{thmtools}
\usepackage[shortlabels]{enumitem}
\usepackage{braket}
\usepackage{comment}
\usepackage[hidelinks]{hyperref}
\usepackage{cleveref}
\usepackage{color}
\usepackage[margin=2.5cm]{geometry}

\newcommand{\sech}{\operatorname{sech}}

\newcommand{\IC}{\mathbb C}
\newcommand{\IN}{\mathbb N}
\newcommand{\IR}{\mathbb R}

\newcommand{\cH}{\mathcal H}
\newcommand{\cJ}{\mathcal J}
\newcommand{\cL}{\mathcal L}
\newcommand{\cV}{\mathcal V}

\newcommand{\id}{\mathrm{id}}
\newcommand{\tr}{\operatorname{tr}}

\newcommand{\abs}[1]{\lvert#1\rvert}
\newcommand{\norm}[1]{\lVert#1\rVert}
\newcommand{\ip}[1]{\langle #1\rangle}

\renewcommand{\Re}{\operatorname{Re}}
\renewcommand{\Im}{\operatorname{Im}}

\renewcommand{\epsilon}{\varepsilon}
\renewcommand{\phi}{\varphi}

\newcommand{\Tr}{\mathrm{Tr}}

\newcommand{\dom}{\operatorname{dom}}

\newcommand{\binotimes}{\mathbin{\otimes_{\mathrm{bin}}}}

\newtheorem{lemma}{Lemma}[section]
\newtheorem{theorem}[lemma]{Theorem}
\newtheorem{corollary}[lemma]{Corollary}
\newtheorem{proposition}[lemma]{Proposition}

\newtheorem{conjecture}[lemma]{Conjecture}

\theoremstyle{definition}
\newtheorem{definition}[lemma]{Definition}

\theoremstyle{remark}
\newtheorem{remark}[lemma]{Remark}
\newtheorem{example}[lemma]{Example}

\title[KMS-symmetric Markov generators]{The Differential Structure of Generators of KMS-Symmetric Quantum Markov Semigroups}
\author{Matthijs Vernooij}
\address[MV]{Research and Development team, Enexis Group, Orthen 63, 5231 XP 's Hertogenbosch, The Netherlands\newline
Faculty of Electrical Engineering, Mathematics and Computer Science, Delft University of Technology, The Netherlands
}
\email{matthijs.vernooij@enexis.nl}
\author{Melchior Wirth}
\address[MW]{Faculty of Mathematics and Computer Science, Leipzig University, Neues Augusteum, Augustusplatz 10, 04109 Leipzig, Germany}
\email{melchior.wirth@uni-leipzig.de}
\date{}

\begin{document}

\begin{abstract}
    A KMS-symmetric quantum Markov semigroup is completely determined by the quadratic form in induces on the GNS Hilbert space. The quadratic forms that arise in this way are called quantum Dirichlet forms. In this article we give a complete description of bounded quantum Dirichlet forms in terms of derivations. In particular, we exhibit an explicit characterization of the bounded quantum Dirichlet forms on type $\mathrm{I}$ factors. We also show that any unbounded quantum Dirichlet form can be represented in terms of derivations. The key technical insight is that the modular group always restricts to a strongly continuous group on the domain of a quantum Dirichlet form such that the square root of the analytic generator is accretive with respect to the form norm.
\end{abstract}

\maketitle

\section{Introduction}

Quantum Markov semigroups have been first introduced and studied in mathematical physics as models of the time evolution of open quantum systems \cite{Lin76,GKS76,AL87}. Since then, they have become objects of intense research with importance to both mathematical physics and pure mathematics -- we refer to \cite{Fag99,Cip08,Ska19,Cip23} for some surveys.

As one-parameter semigroups of operators, quantum Markov semigroups are determined by their derivative at time zero, the generator. It is the superoperator that governs what is called the Lindblad master equation for open quantum systems in the physics literature. In their seminal work, Lindblad \cite{Lin76} and Gorini--Kossakowski--Sudarshan \cite{GKS76} completely characterized the generators of quantum Markov semigroups on matrix algebras and more generally of uniformly continuous quantum Markov semigroups on hyperfinite von Neumann algebras $M$ as operators of the form $\cL(x)=k^\ast x+xk-\Psi(x)$ with $k\in M$ and $\Psi\colon M\to M$ a normal completely positive map. Building on Christensen's deep work on the derivation problem, this characterization was extended to the generators of quantum Markov semigroups on arbitrary von Neumann algebras by Christensen and Evans \cite{CE79}.

One theme that is already present in the work of Lindblad and Christensen--Evans is the close relation between Markov generators and derivations. This theme has been further pursued, most notably by Sauvageot and Cipriani \cite{Sau89,Sau90,CS03b}. It takes probably its clearest form when the quantum Markov semigroup satisfies a detailed balance or symmetry condition, as is often the case for open quantum systems coupled to environments in equilibrium.

In this case, the quantum Markov semigroup can be extended to the GNS Hilbert space, which opens up the powerful toolbox of self-adjoint operators on a Hilbert space. Symmetric quantum Markov semigroups give rise to strongly continuous symmetric contraction semigroups, whose generator is a positive self-adjoint operator and is thus characterized by its induced quadratic form. The quadratic forms on the GNS Hilbert space that arise in this way are called conservative completely Dirichlet forms or quantum Dirichlet forms. Cipriani and Sauvageot gave a characterization of quantum Dirichlet forms in terms of derivations in the case when the reference state is a trace. This result has been influential across a wide range of fields, including analysis on fractals \cite{IRT12,HT13}, deformation and rigidity theory for von Neumann algebras \cite{Pet09a,Pet09b,OP10,DI16} and noncommutative optimal transport and Ricci curvature bounds \cite{BGJ22,BGJ23,Wir20,WZ21}.

However, while the condition of tracial symmetry is very natural in the context of noncommutative analysis, it has its limitations, especially when it comes to physical applications. For typical physical models the invariant states are Gibbs states of a reduced Hamiltonian, which are only traces at infinite temperature.  

For this reason, it is natural to consider also quantum Markov semigroups that are symmetric with respect to non-tracial states. This makes the theory richer, but also more complicated. For one, given a non-tracial state, there is not only one induced inner product, but a whole family of inner products and corresponding symmetry conditions, among them as the most important ones for applications the GNS and KMS symmetry.

A completely positive map $\Phi$ on a von Neumann algebra is called GNS-symmetric with respect to the faithful normal state $\phi$ if
\begin{equation*}
    \phi(\Phi(x)^\ast y)=\phi(x^\ast\Phi(y)),\quad x,y\in M,
\end{equation*}
and KMS-symmetric if
\begin{equation*}
    \langle \Phi(x)\Omega_\phi,\Delta_\phi^{1/2}y \Omega_\phi\rangle=\langle x\Omega_\phi,\Delta_\phi^{1/2}\Phi(y)\Omega_\phi\rangle,\quad x,y\in M,
\end{equation*}
where $\Omega_\phi$ the cyclic vector in the GNS representation and $\Delta_\phi$ is the modular operator.

For GNS-symmetric quantum Markov semigroups, a characterization of the associated quantum Dirichlet forms in terms of derivations that parallels the tracial case was given by the second-named author first in the bounded case \cite{Wir24a} and later extended to the unbounded case \cite{Wir22b,Wir24b}. The key ingredient is the following remarkable, if long known, fact: If a quantum Markov semigroup is GNS-symmetric, it commutes with the modular group of the reference state. In particular, the class of GNS-symmetric quantum Markov semigroups is a (strict) subset of the class of KMS-symmetric quantum Markov semigroups for a given reference states.

In terms of the associated quantum Dirichlet form $\mathcal E$, this commutation with the modular group is equivalent to saying that $\mathcal E$ is invariant under $(\Delta_\phi^{is})$, or, in more semigroup-theoretic terms, the group $(\Delta_\phi^{is})_{s\in\IR}$ restricts to a strongly continuous group $(V_s)_{s\in\IR}$ on $(\dom(\mathcal E),\norm\cdot_{\mathcal E})$ such that $V_{-i/2}$, the square root of the analytic generator, is a positive self-adjoint operator (the reason for formulating this result in terms of the square root of the analytic generator instead of the analytic generator itself will become apparent once we discuss the KMS-symmetric case).

On the side of KMS-symmetric quantum Markov semigroups, Ko \cite{Ko07} obtained an integral representation for bounded generators in terms of derivations. In \cite{VW23}, the authors also obtained a representation theorem for quantum Dirichlet forms associated with KMS-symmetric quantum Markov semigroups. The crucial new ingredient is the introduction and analysis of the $\cV$-transform, the solution operator for the operator equation $\frac 1 2\Delta_\phi^{1/4}X \Delta_\phi^{-1/4}+\frac 1 2\Delta_\phi^{-1/4}X\Delta_\phi^{1/4}=Y$, which occurs naturally when one tries to generalize approach from \cite{CS03a,Wir22b} to KMS-symmetric quantum Markov semigroups. If $X$ commutes with the modular operator, then clearly the solution is simply given by $X=Y$, so the $\cV$-transform is a genuinely new feature that only appears for KMS-symmetric semigroups that are not GNS-symmetric.

However, one drawback of the result from \cite{Ko07,VW23} is that they do not give a characterization of quantum Dirichlet forms, and not only for technical domain reasons, but for purely algebraic ones: Even on matrix algebras, there are derivations in the sense of \cite{VW23} such that the form $Q(A)=\norm{\delta(A)}^2$ is not a quantum Dirichlet form. Likewise, Ko did not give necessary and sufficient conditions for his derivations to give rise to a quantum Dirichlet form, although the results by Park discussed below come close to establishing the converse direction.

To obtain this converse implication in the GNS-symmetric case in \cite{Wir24b}, it was important to only consider derivations that reflect the invariance of the quantum Dirichlet form under the modular group. The failure of the converse of the representation result from \cite{VW23} therefore suggests that the quantum Dirichlet forms associated with KMS-symmetric quantum Markov semigroups also possess a, so far undiscovered,``hidden symmetry'' property in relation to the modular group.

In this article, we show that this is indeed the case. More precisely, in analogy to the GNS-symmetric case, we prove that for a quantum Dirichlet form $\mathcal E$ associated with a KMS-symmetric quantum Markov semigroup, the group $(\Delta_\phi^{is})_{s\in\IR}$ restricts to a strongly continuous group $(V_s)_{s\in\IR}$ on $(\dom(\mathcal E),\norm\cdot_{\mathcal E})$ such that $V_{-i/2}$ is accretive. To highlight the difference to the GNS-symmetric case, let us note that $V_{-i/2}$ is positive self-adjoint if and only if $\langle a,V_{-i/2}a\rangle\geq 0$ for all $a\in \dom(V_{-i/2})$, whereas it is accretive if and only if $\Re\langle a,V_{-i/2}a\rangle\geq 0$ for all $a\in \dom(V_{-i/2})$.

Based on this observation, we refine the results from \cite{Ko07,VW23} to give a complete characterization of bounded quantum Dirichlet forms in terms of derivations. Using an ultraproduct construction, we show that an analogous representation in terms of derivations also holds for unbounded quantum Dirichlet forms. Let us point out that in contrast to prior work (for example \cite{Ko09}), we do not make any additional assumptions on the domain of the generator of the quantum Markov semigroup.

Of particular interest in mathematical physics is the case when the underlying von Neumann algebra is a type $\mathrm{I}$ factor, i.\,e. of the form $B(H)$ for some Hilbert space $H$. In this context, Fagnola and Umanità \cite{FU10} showed that the parts $x\mapsto k^\ast x+xk$ and $\Psi$ in the Lindblad representation $\cL(x)=k^\ast x+xk-\Psi(x)$ of the generator can both be chosen KMS-symmetric. Amorim and Carlen \cite{AC21} refined this result (for $H$ finite-dimensional) by showing that the generators can in fact be parametrized by the cp map $\Psi$ alone. We will revisit these results here.

On the side of quadratic forms, Park \cite{Par00} gave the following construction of bounded quantum Dirichlet forms on standard forms of von Neumann algebras: Let us say a continuous function $f$ on the closure of the strip $S=\{z\in\IC:\abs{\Im z}<1/4\}$ is \emph{admissible} if $f$ is analytic on $S$, $f(t)\geq 0$ for $t\in\IR$, $f(t+i/4)+f(t-i/4)\geq 0$ for $t\in \IR$ and there exist $C>0$, $p>1$ such that $\abs{f(t+is)}\leq C(1+\abs t)^{-p}$ for all $t+is\in \overline S$. For example, $f(z)=\frac{2}{\cosh 2\pi z}$ satisfies these conditions.

If $x\in \dom(\sigma^\phi_{-i/4})\cap\dom(\sigma^\phi_{i/4})$ is self-adjoint, then
\begin{equation*}
    \mathcal E(a)=\int_\IR \norm{\sigma^\phi_{-i/4}(x)(\Delta_\phi^{it}a)-(\Delta_\phi^{it}a)\sigma^\phi_{i/4}(x)}_2^2 f(t)\,dt
\end{equation*}
defines a bounded quantum Dirichlet form on the standard Hilbert space of a von Neumann algebra. Using approximation techniques, this result has later been extended to construct unbounded quantum Dirichlet forms \cite{BK02}.

In this article, we show that the Dirichlet forms constructed by Park do not only provide a rich class of examples, but in fact exhaust the bounded quantum Dirichlet forms on $B(H)$ in the sense that any bounded quantum Dirichlet form is a (finite or infinite) sum of quantum Dirichlet forms of the type constructed by Park, and one can in fact always take the specific admissible function $f(z)=\frac{2}{\cosh 2\pi z}$. The accretivity of the operator $V_{-i/2}$ mentioned above is reflected in the condition $f(t+i/4)+f(t-i/4)\geq 0$ for an admissible function. 

Our results also have some remarkable technical consequences on approximation properties for quantum Dirichlet forms. In the tracial case, it was shown by Davies and Lindsay \cite{DL92} that the bounded vectors in the form domain form a $\ast$-algebra, and as a consequence of the Dirichlet form property, they are a form core. This was refined by the second-named author in the GNS-symmetric case \cite{Wir22b}, where he showed that the intersection of the form domain with the maximal Tomita algebra associated with the state is a Tomita algebra and a form core. The proof relied heavily on the invariance of the quantum Dirichlet form under the modular group in the GNS-symmetric case, but we show here that the result continues to hold in the KMS-symmetric case. These results underscore the usefulness of quadratic form techniques in the study of symmetric quantum Markov semigroups -- in contrast, it may happen that the largest $\ast$-subalgebra contained in the domain of the generator reduces to multiples of the identity.

A final question we address is that of embedding the codomain of the derivation arising from a quantum Dirichlet form into the $L^2$-space of a larger von Neumann algebra with a faithful normal conditional expectation onto the base algebra. For uniformly continuous (or more generally so-called $\Gamma$-regular ones) GNS-symmetric quantum Markov semigroups, this can always be achieved using an operator-valued free Araki--Woods construction \cite{JRS,Wir22b,KW25}. One key ingredient is that the invariance properties of quantum Dirichlet forms associated with GNS-symmetric quantum Markov semigroups give rise to ``proto modular data'' (termed a Tomita correspondence in \cite{Wir22b}), which ensures that the natural conditional expectation in Shlyakhtenko's operator-valued semicircular construction is faithful.

While this invariance under the modular group does not necessary hold in the KMS-symmetric case, the accretivity of the operator $V_{-i/2}$ discussed above gives rise to a strongly continuous group $(\cV_s)_{s\in\IR}$ on the codomain of the derivation such that $\cV_{-i/2}$ is accretive. We show that the group $(\cV_s)_{s\in\IR}$ can be deformed into a strongly continuous unitary group so that the codomain of the derivation gets the structure of a Tomita correspondence.

The structure of the article is as follows: In \Cref{sec:prelim} we collect some necessary preliminaries. In \Cref{sec:char_bdd} we prove an integral representation theorem (\Cref{thm:int_rep_form}) and give a full characterization of bounded quantum Dirichlet forms in terms of derivations (\Cref{thm:char_bdd}). In \Cref{sec:type_I} we focus on type $\mathrm{I}$ factors and show that in this case, the bounded quantum Dirichlet forms are exactly the (finite or infinite) sums of the Dirichlet forms constructed by Park.

In \Cref{sec:Tomita_alg} we explore some consequences of our characterization of bounded quantum Dirichlet forms for the domains of unbounded ones. In particular, we show that every quantum Dirichlet form $\mathcal E$ gives rise to a canonical Tomita algebra $\mathfrak A_{\mathcal E}$, which is a form core for $\mathcal E$, and we show that the square root of the analytic generator of the modular group, viewed as strongly continuous group on the domain of $\mathcal E$ with the form norm, is an accretive operator (\Cref{thm:analytic_gen_mod_grp_accretive}). In \Cref{sec:unbdd_Dir_forms} we use ultraproduct techniques to extend our result from the bounded case to give a representation of general quantum Dirichlet forms in terms of derivations (\Cref{thm:char_unbdd}). Finally, in \Cref{sec:Tomita_corr}, we show how the codomain of the derivation induced by a bounded quantum Dirichlet form gives rise to a Tomita correspondence and discuss the consequences of the operator-valued semicircular construction in this case (\Cref{thm:Tomita_corr_Dir_form}).

\subsection*{Acknowledgments} M.W. gratefully acknowledges funding by the Deutsche Forschungsgemeinschaft
(DFG, German Research Foundation) – Project number 565405322.

\section{Preliminaries}\label{sec:prelim}

In this section we collect some preliminaries to be used later in the article. The material here is well-known or follows easily from well-known facts, with the exception of \Cref{prop:KMS_state_inv}, which shows that a quantum Markov semigroup is KMS-symmetric with respect to a given (faithful normal) invariant state if and only if it is KMS-symmetric with respect to every (faithful normal) invariant state.

\subsection{Quadratic forms}\label{sec:quad_forms}

Let $H$ be a Hilbert space. A \emph{quadratic form} on $H$ is a map $Q\colon H\to [0,\infty]$ such that $Q(\xi+\eta)+Q(\xi-\eta)=2Q(\xi)+2Q(\eta)$ and $Q(\lambda \xi)=\abs{\lambda}^2Q(\xi)$ for all $\xi,\eta\in H$ and $\lambda\in \IC$. The set $\{\xi\in H\mid Q(\xi)<\infty\}$ is called the \emph{domain} of $Q$ and denoted by $\dom(Q)$. A quadratic form $Q$ is called \emph{densely defined} if $\dom(Q)$ is dense in $H$, \emph{closed} if it is lower semicontinuous and \emph{closable} if its restriction to $\dom(Q)$ is lower semicontinuous. If $Q$ is closable, then the lower semicontinuous relaxation
\begin{equation*}
    \overline Q\colon H\to [0,\infty],\,\xi\mapsto \liminf_{\eta\to \xi}Q(\eta)
\end{equation*}
is a closed quadratic form that coincides with $Q$ on $\dom(Q)$. In this case, the form $\overline Q$ is called the \emph{closure} of $Q$.

The \emph{form norm} $\norm\cdot_Q$ on $\dom(Q)$ is defined as $\norm{\cdot}_Q=(\norm{\cdot}_H^2+Q(\cdot))^{1/2}$. The form $Q$ is closed if and only if $(\dom(Q),\norm\cdot_Q)$ is a Banach space. A subset $D$ of $\dom(Q)$ is called a \emph{form core} for $Q$ if it is dense in $\dom(Q)$ with respect to the form norm.

By the polarization identity, there is a unique hermitian sesquilinear from on $\dom(Q)\times \dom(Q)$ that coincides with $Q$ on the diagonal. By a slight abuse of notation, we also denote this sesquilinear form by $Q$. By \cite[Proposition 10.5, Corollary 10.8]{Sch12}, if $L$ is a positive self-adjoint operator in $H$, the quadratic form $Q_L$ characterized by $\dom(Q_L)=\dom(L^{1/2})$ and $Q_L(\xi)=\norm{L^{1/2}\xi}^2$ for $\xi\in\dom(L^{1/2})$ is closed and densely defined, and the map $L\mapsto Q_L$ is a bijection from the set of positive self-adjoint operators in $H$ onto the set of closed densely defined quadratic forms on $H$.

A quadratic form $Q$ is called \emph{bounded} if there exists $C>0$ such that $Q(\xi)\leq C\norm{\xi}^2$ for all $\xi\in H$. The quadratic form $Q$ is bounded if and only if there exists a bounded self-adjoint operator $L$ on $H$ such that $Q=Q_L$. If $Q$ is a closed densely defined quadratic form on $H$ and $\epsilon>0$, then the form $Q^{(\epsilon)}$ defined by
\begin{equation*}
    Q^{(\epsilon)}(\xi)=\inf_{\eta\in H} Q(\eta)+\frac 1 \epsilon \norm{\xi-\eta}^2
\end{equation*}
is a bounded quadratic form, called the \emph{Moreau regularization} of $Q$. For every $\xi\in H$ one has $Q^{(\epsilon)}(\xi)\nearrow Q(\xi)$ as $\epsilon\searrow 0$. If $L$ is a positive self-adjoint operator in $H$, then $Q_L^{(\epsilon)}=Q_{L(1+\epsilon L)^{-1}}$, as can be easily deduced from \cite[Theorem 10.9]{Sch12}.

\subsection{Standard forms and modular theory}

In this subsection, we briefly discuss Haagerup's theory of standard forms of von Neumann algebras \cite{Haa75} and some basic aspects of modular theory. A quadruple $(M,H,J,P)$, where $H$ is a Hilbert space, $M$ is a von Neumann algebra on $H$, $J\colon H\to H$ is an anti-unitary involution and $P\subset H$ is a self-dual cone, is called a \emph{standard form} if it satisfies the following three conditions:
\begin{enumerate}[(a)]
    \item $JMJ=M^\prime$,
    \item $J \xi=\xi$ for all $\xi\in P_+$,
    \item $x JxJ (P)\subset P$ for all $x\in M$.
\end{enumerate}
The original definition by Haagerup contained the fourth condition $JxJ=x^\ast$ for $x\in \mathcal Z(M)$, but this was later shown to be redundant \cite[Lemma 3.19]{AH14}.

For every von Neumann algebra $M$ there exists a Hilbert space $H$, an anti-unitary involution $J\colon H\to H$, a self-dual cone $P\subset H$ and a faithful normal representation $\pi\colon M\to B(H)$ such that $(\pi(M),H,J,P)$ is a standard form \cite[Theorem 1.6]{Haa75}. Moreover, if $(M_k,H_k,J_k,P_k)$ is a standard form for $k\in \{1,2\}$ and $\alpha\colon M_1\to M_2$ is a $\ast$-isomorphism, then by \cite[Theorem 2.3]{Haa75} there exists a unique unitary $U\colon H_1\to H_2$ such that
\begin{itemize}
    \item $\alpha(x)=UxU^\ast$ for all $x\in M_1$,
    \item $J_2 U=UJ_1$,
    \item $U P_1=P_2$.
\end{itemize}
In this sense, every von Neumann algebra $M$ has an essentially unique standard form, which will be denoted by $(M,L^2(M),J,L^2_+(M))$ in this article.

The standard form of the von Neumann algebra tensor product has been characterized in \cite{MT84}. We only need the case when one of the algebras is a finite type $\mathrm I$ factor.

\begin{proposition}[{\cite[Proposition 2.4]{MT84}}]\label{prop:std_form_matrix_ampl}
    For $n\in\IN$ let $S^2_n$ denote the complex $n\times n$ matrices with the Hilbert--Schmidt inner product, let $J_n\colon S^2_n\to S^2_n$, $A\mapsto A^\ast$ and view $M_n(\IC)$ as subalgebra of $B(S^2_n)$ via left multiplication. If $(M,H,J,P)$ is a standard form and $P^{(n)}=\{\sum_{j,k=1}^n \xi_{jk}\otimes E_{jk}\in H\otimes S^2_n\mid \sum_{j,k=1}^n x_j J x_k J\xi_{j,k}\in P\}$, then $(M\otimes M_n(\IC),H\otimes S^2_n,J\otimes J_n,P^{(n)})$ is a standard form.
\end{proposition}

The subset $L^2(M)^J=\{\xi\in L^2(M)\mid J\xi=\xi\}$ is a real Hilbert space. Since $L^2_+(M)\subset L^2(M)^J$ is a closed convex subset, there exists for every $\xi\in L^2(M)^J$ a unique element $\xi_+\in L^2_+(M)$ such that $\norm{\xi-\xi_+}=\min_{\eta\in L^2_+(M)}\norm{\xi-\eta}$. We write $\xi_-$ for $(-\xi)_+$ and note that $\xi=\xi_+-\xi_-$ for all $\xi\in L^2(M)^J$.

If $\phi\in M_\ast^+$, then there exists a unique vector $\Omega_\phi\in L^2_+(M)$ such that $\phi(x)=\langle \Omega_\phi,x\Omega_\phi\rangle$ for all $x\in M$ \cite[Lemma 2.10]{Haa75}. A functional $\phi\in M_\ast^+$ is faithful if and only if $\Omega_\phi$ is cyclic and separating for $M$. In this case, the quadratic form
\begin{equation*}
    Q_\phi^{(0)}\colon M\Omega_\phi\subset L^2(M)\to [0,\infty),\,x\Omega_\phi\mapsto \norm{x^\ast \Omega_\phi}^2
\end{equation*}
is well-defined and closable. If $Q_\phi$ denotes its closure, then the unique positive self-adjoint operator $\Delta_\phi$ in $L^2(M)$ such that $\dom(\Delta_\phi^{1/2})=\dom (Q_\phi)$ and $\norm{\Delta_\phi^{1/2}\xi}^2=Q_\phi(\xi)$ for all $\xi\in \dom(Q_\phi)$ is called the \emph{modular operator} of $\phi$.

By Tomita's theorem, $\Delta_\phi^{it}M\Delta_\phi^{-it}=M$ for all $t\in \IR$. The family $\sigma^\phi$ given by $\sigma^\phi_t(x)=\Delta_\phi^{it}x\Delta_\phi^{-it}$ for $x\in M$ and $t\in\IR$ is called the \emph{modular group} of $\phi$. An element $x\in M$ is called \emph{entire analytic} for $\sigma^\phi$ if $t\mapsto \sigma^\phi_t(x)$ has an extension to an entire function. In this case, this extension is unique by the identity theorem for analytic functions and denoted by $z\mapsto\sigma^\phi_t(x)$. The set $\{x\Omega_\phi\mid x\text{ entire analytic for }\sigma^\phi\}$ is called the \emph{maximal Tomita algebra associated with $\phi$}.

If $x\in M$, one defines $\pi_l(x\Omega)=x$, and if $y^\prime\in M^\prime$, one defines $\pi_r(y^\prime \Omega_\phi)=y^\prime$. Furthermore, if $a\in M\Omega_\phi$ and $b\in L^2(M)$ (resp. $a\in L^2(M)$ and $b\in M^\prime \Omega_\phi$), one defines $ab=\pi_l(a)b$ (resp. $ab=\pi_r(b)a$). This product is consistent in the sense that $\pi_l(a)b=\pi_r(b)a$ if $a\in M\Omega_\phi$ and $b\in M^\prime\Omega_\phi$. It does depend on the choice of the state $\phi$, but since the state will be fixed in the main part of the article, this should not cause any confusion.

A subset $\mathfrak A$ of the maximal Tomita algebra associated with $\phi$ is a \emph{Tomita subalgebra} if it is closed under the maps $J$, $\Delta_\phi^z$ for $z\in \IC$ and the product described in the previous paragraph.

\subsection{KMS-symmetric completely positive maps}

Let $A,B$ be $C^\ast$-algebras. Recall that a linear map $\Phi\colon A\to B$ is called \emph{positive} if $\Phi(A_+)\subset \Phi(B_+)$ and \emph{completely positive} if $\Phi\otimes\id((A\otimes M_n(\IC))_+)\subset (B\otimes M_n(\IC))_+$ for every $n\in \IN$. If $A$ and $B$ are unital, then $\Phi$ is called \emph{unital} if $\Phi(1)=1$. We simply write \emph{(u)cp} for (unital) completely positive

If $M$ is a von Neumann algebra and $\phi\in M_\ast^+$ is a faithful normal state, then a linear map $\Phi\colon M\to M$ is called \emph{KMS-symmetric (w.\,r.\,t. $\phi$)} if 
\begin{equation*}
    \langle x\Omega_\phi,\Delta_\phi^{1/2}\Phi(y)\Omega_\phi\rangle=\langle \Phi(x)\Omega_\phi,\Delta_\phi^{1/2} y\Omega_\phi\rangle
\end{equation*}
for all $x,y\in M$. Every KMS-symmetric linear map is weak$^\ast$ continuous by \cite[Proposition 2.3]{Cip97}. Moreover, if $\Phi$ is KMS-symmetric and unital, then $\phi\circ\Phi=\phi$ as a direct consequence of the definition.

Following \cite[Section 1]{Cip97}, we define embeddings $i_\phi\colon M\to L^2(M)$ and $j_\phi\colon M\to M_\ast$ by $i_\phi(x)=\Delta_\phi^{1/4}(x\Omega_\phi)$ and $j_\phi(x)(y)=\langle J x\Omega_\phi,y\Omega_\phi\rangle$. In other words, $j_\phi=i_\phi^\ast J i_\phi$ if we identify the dual of a Hilbert space with the conjugate Hilbert space.

A weak$^\ast$ continuous linear map $\Phi\colon M\to M$ is KMS-symmetric with respect to $\phi$ if and only if $j_\phi\circ \Phi=\Phi_\ast\circ j_\phi$ \cite[Lemma 2.2]{Cip97}. If $\Phi\colon M\to M$ is KMS-symmetric with respect to $\phi$, then there exists a unique bounded linear operator $T\in B(L^2(M))$ such that $T\circ i_\phi=i_\phi\circ \Phi$ \cite[Lemma 2.2, Proposition 2.4]{Cip97}. In the following, we write $\Phi^{(2,\phi)}$ or simply $\Phi^{(2)}$ for $T$ and call it the \emph{KMS implementation} of $\Phi$.

If $\Phi$ is KMS-symmetric with respect to $\phi$, then it is easily seen that $\Phi\otimes \id_{M_n(\IC)}$ is KMS-symmetric with respect to $\phi\otimes \tr$. Moreover, if we identity $L^2(M\otimes M_n(\IC))$ with $L^2(M)\otimes S^2_n$ as in \Cref{prop:std_form_matrix_ampl}, then $(\Phi\otimes \id_{M_n(\IC)})^{(2,\phi\otimes\tr)}=\Phi^{(2)}\otimes \id_{S^2_n}$.

We call a map $T\colon L^2(M)\to L^2(M)$ \emph{positivity-preserving} if $T(L^2_+(M))\subset L^2_+(M)$ and \emph{completely positivity-preserving} if $T\otimes \id(P^{(n)})\subset P^{(n)}$ for all $n\in\IN$, where $P^{(n)}\subset L^2(M)\otimes S^2_n$ is the positive cone described in \Cref{prop:std_form_matrix_ampl}.

The map $\Phi\mapsto \Phi^{(2,\phi)}$ is a bijection between the positive KMS-symmetric maps on $M$ and the symmetric positivity-preserving operators on $L^2(M)$. Moreover, $\Phi$ is unital if and only if $\Phi^{(2,\phi)}(\Omega_\phi)=\Omega_\phi$ and $\Phi$ is completely positive if and only $\Phi^{(2)}$ is completely positivity-preserving. The first two facts follow from \cite[Theorems 2.11, 2.12]{Cip97}, while the last one is an immediate consequence of the description of $(\Phi\otimes\id_{M_n(\IC)})^{(2,\phi\otimes\tr)}$ given above.

\subsection{Quantum Markov semigroups and quantum Dirichlet forms}

Let $M$ be a von Neumann algebra. A family $(\Phi_t)_{t\geq 0}$ of linear maps on $M$ is called a \emph{quantum Markov semigroup} if it satisfies the following conditions:
\begin{itemize}
    \item $\Phi_t$ is a normal ucp map for every $t\geq 0$,
    \item $\Phi_0=\id_M$, $\Phi_{s+t}=\Phi_s\Phi_t$ for all $s,t\geq 0$,
    \item $\Phi_t(x)\to x$ in weak$^\ast$ topology as $t\to 0$ for every $x\in M$.
\end{itemize}
The quantum Markov semigroup $(\Phi_t)_{t\geq 0}$ is called \emph{uniformly continuous} if $\norm{\Phi_t-\id_M}\to 0$ as $t\to 0$. Moreover, if $\phi$ is a faithful normal state on $M$, then $(\Phi_t)_{t\geq 0}$ is called \emph{KMS-symmetric} (w.\,r.\,t. $\phi$) if $\Phi_t$ is KMS-symmetric w.\,r.\,t. $\phi$ for every $t\geq 0$.

The fact that we usually drop $\phi$ from the notation is justified by the following result, which is a KMS analog of \cite[Lemma 4.16]{GJLL25} for GNS-symmetric maps.

\begin{proposition}\label{prop:KMS_state_inv}
    Let $(\Phi_t)_{t\geq 0}$ be a quantum Markov semigroup on $M$. If $\phi$, $\psi$ are faithful normal states on $M$ that are invariant under $(\Phi_t)_{t\geq 0}$ and $(\Phi_t)_{t\geq 0}$ is KMS-symmetric with respect to $\phi$, then it is also KMS-symmetric with respect to $\psi$.
\end{proposition}
\begin{proof}
    Let $N=\{x\in M\mid \Phi_t(x)=x\text{ for all }t\geq 0\}$. By \cite[Theorem A]{Wat79}, the set $N$ is a von Neumann subalgebra of $M$, for every $x\in M$ the limit $E(x)=\lim_{t\to\infty}\frac 1t\int_0^t \Phi_s(x)\,ds$ exists in the strong operator topology and $E$ is a faithful normal conditional expectation such that $\phi\circ E=\phi$ and $\psi\circ E=\psi$. In particular, $N$ is invariant under $\sigma^\phi$ and $\sigma^\psi$ by Takesaki's theorem \cite[Theorem IX.4.2]{Tak03}.

    Let $H_N=\overline{N\Omega_\phi}\subset L^2(M)$. Since $N$ is invariant under $\sigma^\phi$, the subspace $H_N$ is invariant under $\Delta_\phi^{it}$ for all $t\in\IR$. Thus, if $x\in N$, then $J(x\Omega_\phi)=\Delta_\phi^{1/2}(x^\ast \Omega_\phi)\in H_N$. By continuity, it follows that $J(H_N)\subset H_N$. Let $J_N=J|_{H_N}$ and $P_N=\overline{\{x J_N x J_N \Omega_\phi\mid x\in N\}}$. 
    
    By \cite[Remark 1.2]{Haa75}, the quadruple $(N,H_N,J_N,P_N)$ is a standard form. Therefore, there exists a unique $\xi\in P_N$ such that $\psi(x)=\langle \xi,x\xi\rangle$ for all $x\in N$ by \cite[Lemma 2.10]{Haa75}. Let $h_n\in N$ such that $h_n\Omega_\phi\to \xi$. For $y\in M$ we have
    \begin{align*}
        \langle \xi,y\xi\rangle=\lim_{n\to \infty}\phi(h_n^\ast yh_n)=\lim_{n\to\infty}\phi(h_n^\ast E(y)h_n)=\langle \xi,E(y)\xi\rangle=\psi(E(y))=\psi(y).
    \end{align*}
    Since $P_N\subset L^2_+(M)$, we conclude that $\xi=\Omega_\psi$. Now we get 
    \begin{align*}
        \langle \Phi_t(x)\Omega_\psi,\Delta_\psi^{1/2}y \Omega_\psi\rangle&
        =\langle J y\Omega_\psi,\Phi_t(x)^\ast\Omega_\psi\rangle\\
        &=\lim_{n\to\infty} \langle \Phi_t(x)J y J h_n \Omega_\phi,h_n\Omega_\phi\rangle\\
        &=\lim_{n\to\infty}\langle \Phi_t(h_n^\ast x h_n)Jy J\Omega_\phi,\Omega_\phi\rangle\\
        &=\lim_{n\to\infty}\langle\Phi_t(h_n^\ast x h_n)\Omega_\phi,\Delta_\phi^{1/2}y\Omega_\phi\rangle\\
        &=\lim_{n\to\infty}\langle h_n^\ast xh_n\Omega_\phi,\Delta_\phi^{1/2}\Phi_t(y)\Omega_\phi\rangle\\
        &=\langle x\Omega_\psi,\Delta_\psi^{1/2}\Phi_t(y)\Omega_\psi\rangle
    \end{align*}
    for all $x,y\in M$.
\end{proof}


The \emph{generator} of a quantum Markov semigroup $(\Phi_t)_{t\geq 0}$ is the weak$^\ast$ closed and densely defined operator $\mathcal L$ in $M$ defined by
\begin{align*}
    \dom(\cL)&=\left\{x\in M\mid \lim_{t\to 0}\frac 1 t(x-\Phi_t(x))\text{ exists in weak}^\ast\text{ topology}\right\},\\
    \cL(x)&=\text{weak}^\ast\text{-}\lim_{t\to 0}\frac 1 t(x-\Phi_t(x)).
\end{align*}
The quantum Markov semigroup $(\Phi_t)_{t\geq 0}$ is uniformly continuous if and only if its generator is bounded. The generators of uniformly continuous quantum Markov semigroups can be characterized as follows \cite[Theorem 14.7]{EL77b}: A bounded operator $\cL$ on $M$ is the generator of a uniformly continuous quantum Markov semigroup if and only it is weak$^\ast$ continuous, $\cL(x^\ast)=\cL(x)^\ast$ for all $x\in M$, $\cL (1) = 0$ and $\cL$ is \emph{conditionally completely negative}, that is,
\begin{equation*}
\sum_{j=1}^n y_j^\ast\cL(x_j^\ast x_k)y_k\leq 0
\end{equation*}
whenever $x_1, \dots , x_n , y_1, \dots , y_n \in M$ with $\sum_{j=1}^n x_j y_j = 0$.

By \cite[Theorem 2.11, 2.12]{Cip97}, the map $(\Phi_t)_{t\geq 0}\mapsto (\Phi_t^{(2,\phi)})_{t\geq 0}$ is a bijection from the KMS-symmetric quantum Markov semigroups onto the strongly continuous semigroups of symmetric, completely positivity-preserving maps on $L^2(M)$ that leave $\Omega_\phi$ invariant. Moreover, $(\Phi_t)_{t\geq 0}$ is uniformly continuous if and only if $(\Phi_t^{(2,\phi)})_{t\geq 0}$ is uniformly continuous.

Every strongly continuous symmetric contraction semigroup $(T_t)_{t\geq 0}$ on a Hilbert space $H$ is uniquely determined by its generator, which is the unique positive self-adjoint operator $L$ in $H$ such that $T_t=e^{-tL}$ for $t\geq 0$. A positive self-adjoint operator $L$ in $H$ in turn is uniquely determined by its associated quadratic form, which is given by $\dom(Q_L)=\dom(L^{1/2})$ and $Q_L(\xi)=\norm{L^{1/2}\xi}^2$ for $\xi\in \dom(L^{1/2})$.

A \emph{conservative Dirichlet form (w.\,r.\,t. $\Omega_\phi$)} is a closed densely defined quadratic form $\mathcal E$ on the standard Hilbert space $L^2(M)$ such that $\mathcal E(\Omega_\phi)=0$, $\mathcal E(J\xi)=\mathcal E(\xi)$ for all $\xi\in L^2(M)$ and $\mathcal E(\xi_+)\leq \mathcal E(\xi)$ for all $\xi\in L^2(M)^{J}$. A conservative Dirichlet form $\mathcal E$ is called conservative completely Dirichlet form or \emph{quantum Dirichlet form} if the form $\mathcal E^{(n)}$ on $L^2(M)\otimes S^2_n$ defined by
\begin{equation*}
    \mathcal E^{(n)}\left(\sum_{j,k=1}^n \xi_{jk}\otimes E_{jk}\right)=\sum_{j,k=1}^n \mathcal E(\xi_{jk})
\end{equation*}
is a conservative Dirichlet form for all $n\in\IN$. Here $L^2(M)\otimes S^2_n$ is considered as the standard Hilbert space of $M\otimes M_n(\IC)$ in the sense specified in \Cref{prop:std_form_matrix_ampl}.

It follows essentially from \cite[Theorem 4.11]{Cip97} that the map $L\mapsto Q_L$ is a bijection between the generators of strongly continuous semigroups of symmetric positivity-preserving maps on $L^2(M)$ that leave $\Omega_\phi$ invariant and quantum Dirichlet forms on $L^2(M)$. Combined with result on the KMS implementation, we get a bijection between KMS-symmetric quantum Markov semigroups on $M$ and quantum Dirichlet forms on $L^2(M)$.

The following lemma will come in very handy to approximate quantum Dirichlet forms by bounded ones. This is essentially known, but we did not find a direct proof in the literature that does not go through the associated generators, so we decided to include it here. In contrast to the notation introduced in Subsection \ref{sec:quad_forms}, we write $\mathcal E_\epsilon$ for the Moreau regularization of $\mathcal E$ to clearly distinguish it from the amplification $\mathcal E^{(n)}$ defined above.

\begin{lemma}\label{lem:Moreau_reg_Dir}
    If $\mathcal E$ is a quantum Dirichlet form on $L^2(M)$ and $\epsilon>0$, then the Moreau regularization $\mathcal E_\epsilon$ is also a quantum Dirichlet form.
\end{lemma}
\begin{proof}
    Since $0\leq\mathcal E_\epsilon\leq \mathcal E$, the property $\mathcal E_\epsilon(\Omega_\phi)=0$ is immediate. If $\xi\in L^2(M)$, then
    \begin{align*}
        \mathcal E_\epsilon(J\xi)
        &=\inf_{\eta\in L^2(M)}\mathcal E(\eta)+\frac 1 \epsilon\norm{J \xi-\eta}^2\\
        &=\inf_{\eta\in L^2(M)}\mathcal E(J \eta)+\frac 1 \epsilon\norm{J(\xi-\eta)}^2\\
        &=\mathcal E_\epsilon(\xi).
    \end{align*}
    Moreover, if $\xi\in L^2(M)^{J}$, then
    \begin{align*}
        \mathcal E_\epsilon(\xi_+)&=\inf_{\eta\in L^2(M) }\mathcal E(\eta)+\frac 1 \epsilon\norm{\xi_+-\eta}^2\\
        &\leq \inf_{\eta\in L^2(M)}\mathcal E(\eta_+)+\frac 1 \epsilon\norm{\xi_+-\eta_+}^2\\
        &\leq \inf_{\eta\in L^2(M)}\mathcal E(\eta)+\frac 1 \epsilon \norm{\xi-\eta}^2\\
        &=\mathcal E(\xi),
    \end{align*}
    where we used that $\xi\mapsto \xi_+$ is $1$-Lipschitz in the second inequality. Thus $\mathcal E_\epsilon$ is a conservative Dirichlet form.

    Finally, since
    \begin{align*}
        (\mathcal E^{(n)})_\epsilon((\xi_{jk})_{j,k})&=\inf_{\eta\in L^2(M)}\sum_{j,k=1}^n\mathcal E(\eta_{jk})+\frac 1 \epsilon \sum_{j,k=1}^n\norm{\xi_{jk}-\eta_{jk}}^2\\
        &=\sum_{j,k=1}^n\inf_{\eta\in L^2(M)} \mathcal E(\eta)+\frac 1 \epsilon \norm{\xi_{jk}-\eta}^2\\
        &=\sum_{j,k=1}^n \mathcal E_\epsilon(\xi_{jk})\\
        &=(\mathcal E_\epsilon)^{(n)}((\xi_{jk})_{j,k}),
    \end{align*}
    the quadratic form $\mathcal E_\epsilon$ is a completely Dirichlet form.
\end{proof}

\subsection{\texorpdfstring{The $\mathcal V$-transform}{The V-transform}}

The $\mathcal V$-transform was introduced by the authors in \cite{VW23}. Let $M$ be a von Neumann algebra and $\phi\colon M\to \IC$ a faithful normal state. It was shown in \cite[Lemma 3.1]{VW23} that for every $T\in B(L^2(M))$ the integral 
\begin{equation*}
    2\int_0^\infty \Delta_\phi^{1/4}e^{-r\Delta_\phi^{1/2}}T \Delta_\phi^{1/4}e^{-r\Delta_\phi^{1/2}}\,dr
\end{equation*}
converges in weak operator topology and defines a bounded linear operator $\check T$ on $L^2(M)$. The map $\mathcal V\colon B(L^2(M))\to B(L^2(M))$, $T\mapsto \check T$ is called the $\mathcal V$-transform. The $\mathcal V$-transform is a normal unital completely positive map on $B(L^2(M))$ \cite[Proposition 3.3]{VW23}. Moreover, if $T\in B(L^2(M))$ is completely positivity-preserving, commutes with $J$ and leaves $\Omega_\phi$ invariant, the same is true for $\check T$.

The $\mathcal V$-transform is characterized as the unique solution of an operator equation in the following sense \cite[Lemma 3.4]{VW23}: If $T\in B(L^2(M))$, then $R\in B(L^2(M))$ satisfies $R=\check T$ if and only if 
\begin{equation*}
    \frac 1 2\langle \Delta_\phi^{1/4}\xi,R\Delta_\phi^{-1/4}\eta\rangle+\frac1  2\langle \Delta_\phi^{-1/4}\xi,R\Delta_\phi^{1/4}\eta\rangle=\langle \xi,T\eta\rangle
\end{equation*}
for all $\xi,\eta\in \dom(\Delta_\phi^{1/4})\cap \dom(\Delta_\phi^{-1/4})$.

\subsection{Analytic generators}

Let $Y$ be a Banach space and $X=Y^\ast$. A \emph{point-weak$^\ast$ continuous group} on $X$ is a family $(V_s)_{s\in\IR}$ of linear operators on $X$ satisfying the following conditions:
\begin{itemize}
    \item $V_s$ is weak$^\ast$ continuous for every $s\in\IR$,
    \item $V_0=\id_{X}$, $V_{s+t}=V_sV_t$ for all $s,t\in \IR$,
    \item $V_s x\to x$ in weak$^\ast$ topology as $s\to 0$ for every $x\in X$.
\end{itemize}
By \cite[Lemma 1.1 and Proposition 2.1]{CZ76}, the operator $V_s$ is bounded for every $s\in \IR$ and there exist $c>0$ and $\omega\in \IR$ such that $\norm{V_s}\leq c e^{\omega \abs{s}}$ for all $s\in\IR$.

If $X$ is reflexive, then a linear operator is weak$^\ast$ continuous if and only if it is bounded. Moreover, if $X$ is reflexive, every point-weak$^\ast$ continuous group $(V_s)_{s\in\IR}$ is automatically strongly continuous in the sense that $\norm{V_s x-x}\to 0$ as $s\to 0$ for every $x\in X$ \cite[Theorem I.5.8]{EN00}.

For $\alpha\in \IR$ let $S_\alpha=\{z\in \IC: 0\leq \operatorname{sgn}(\alpha)\Im z\leq \abs{\alpha}\}$ and let $D_\alpha$ be the set of all $x\in X$ for which the map $s\mapsto V_s x$ has a weak$^\ast$ continuous extension $F_x\colon \overline S_\alpha\to X$ that is analytic on the interior of $S_\alpha$. If it exists, the extension $F_x$ is unique by the identity theorem for analytic functions.

For $z\in \IC$ we define the operator $V_z$ in $X$ by
\begin{align*}
    \dom(V_z)=D_{\Im z},\quad V_zx=F_x(z)
\end{align*}
The operator $V_{-i}$ is called the \emph{analytic generator} of $(V_s)_{s\in\IR}$.

For every $z\in \IC$ the operator $V_z$ is weak$^\ast$ closed, injective and the family $(V_z)_{z\in\IC}$ satisfies \cite[Theorem 2.4]{CZ76}:
\begin{itemize}
    \item $V_z^{-1}=V_{-z}$ for all $z\in\IC$,
    \item $V_{z+w}=V_z V_w$ for all $z,w\in\IC$ with $(\Im z)(\Im w)\geq 0$.
\end{itemize}

\begin{lemma}\label{lem:Haagerup_approx_groups}
    Let $Y$ be a Banach space and let $(V_s)_{s\in\IR}$ be a point-weak$^\ast$ continuous group on $Y^\ast$. If $x\in Y^\ast$, then for every $n\in\IN$ the weak$^\ast$ integral
    \begin{equation*}
        x_n=\sqrt\frac n \pi\int_\IR e^{-ns^2}V_s x\,ds
    \end{equation*}
    exists, belongs to $\bigcap_{z\in\IC}\dom(V_z)$, satisfies
    \begin{equation*}
        V_zx_n=\sqrt\frac n\pi \int_\IR e^{-n(s-z)^2 }V_sx\,ds
    \end{equation*}
    for all $z\in\IC$ and $x_n\to x$ in weak$^\ast$ topology as $n\to\infty$. Moreover, if $x\in \dom(V_z)$, then 
    \begin{equation*}
        V_z x_n=\sqrt\frac n \pi\int_\IR e^{-ns^2}V_{s+z} x\,ds.
    \end{equation*}
\end{lemma}
\begin{proof}
    Everything except for the last claim is shown in the proof of \cite[Lemma 2.2]{CZ76}. The last claim follows from \cite[Corollary 2.5]{CZ76}.
\end{proof}

If $X$ is reflexive, then the extension of $(V_s)_{s\in \IR}$ to complex parameters can be characterized as follows \cite[Theorem 4.2]{Uit00}:
\begin{align*}
    \dom(V_z)&=\left\{x\in X: \lim_{n\to\infty}\sqrt\frac n\pi\int_\IR e^{-n(s-z)^2}V_s x\,ds\text{ exists}\right\},\\
    V_z \xi&=\lim_{n\to\infty}\sqrt\frac n\pi\int_\IR e^{-n(s-z)^2}V_sx\,ds.
\end{align*}

\begin{lemma}\label{lem:core_analytic_generator}
    Let $X$ be a reflexive Banach space and $(V_s)_{s\in\mathbb R}$ a point-weak$^\ast$ continuous group on $X$. If $z\in \IC$ and $D\subset \dom(V_z)$ is a linear subspace that is dense in $X$ and invariant under $(V_s)_{s\in\IR}$, then $D$ is a core for $V_z$.
\end{lemma}
\begin{proof}
    Let $x\in \dom(V_z)$ and let $(x_k)$ be a sequence in $D$ such that $x_k\to x$ in $H$. By \Cref{lem:Haagerup_approx_groups}, we have $\sqrt\frac n\pi\int_\IR e^{-ns^2}V_s(x_k-x)\,ds\in \dom(V_z)$ and 
    \begin{equation*}
        V_z\left(\sqrt\frac n\pi\int_\IR e^{-ns^2}V_s(x-x_k)\,ds\right)=\sqrt\frac n\pi\int_\IR e^{-n(s-z)^2}V_s(x-x_k)\,ds.
    \end{equation*}
    Moreover, there exist $c,\omega>0$ such that $\norm{V_s}\leq c e^{\omega\abs s}$ for all $s\in\IR$. Thus
    \begin{align*}
        \left\lVert V_z\left(\sqrt\frac n\pi\int_\IR e^{-ns^2}V_s(x_k-x)\,ds\right)\right\rVert&\leq \sqrt\frac n\pi\int_\IR \abs{e^{-n(s-z)^2}}\norm{V_s(x-x_k)}\,ds\\
        &\leq \sqrt\frac n\pi e^{n(\Im z)^2}\norm{x-x_k}\int_\IR e^{-n(s-\Re z)^2} c e^{\omega\abs{s}}\,ds.
    \end{align*}
    As $\int_\IR e^{-n(s-\Re z)^2} c e^{\omega\abs{s}}\,ds<\infty$, it follows that $\left\lVert V_z\left(\sqrt\frac n\pi\int_\IR e^{-ns^2}V_s(x_k-x)\,ds\right)\right\rVert\to 0$ as $k\to \infty$. Combined with the same argument for $z=0$, this implies that 
    \begin{equation*}
        \sqrt\frac n\pi\int_\IR e^{-ns^2}V_s x_k\,ds\to \sqrt\frac n\pi\int_\IR e^{-ns^2}V_s x\,ds
    \end{equation*}
    as $k\to\infty$ in the graph norm of $V_z$.

    Since $\sqrt\frac n\pi\int_\IR e^{-ns^2}V_s x\,ds\to x$ in the graph norm of $V_z$ as a consequence of \Cref{lem:Haagerup_approx_groups}, we conclude that $x$ belongs to the closed linear span of 
    \begin{equation*}
        \left\{\sqrt\frac n\pi\int_\IR e^{-ns^2}V_sy\,ds:y\in D,n\in\mathbb N\right\}
    \end{equation*}
    with respect to the graph norm.
    
    Moreover, as
    \begin{align*}
        \sqrt\frac n\pi\int_\IR e^{-ns^2}\norm{V_s y}_{V_z}\,ds&=\sqrt\frac n\pi\int_\IR e^{-ns^2}(\norm{V_s y}^2+\norm{V_{s+z} y}^2)^{1/2}\,ds\\
        &\leq K\left(\sqrt \frac n\pi\int_\IR e^{-ns^2}e^{\omega\abs{s}}\,ds\right)\norm{y}_{V_z},
    \end{align*}
    the integral $\sqrt\frac n\pi\int_\IR e^{-ns^2}V_s y\,ds$ converges absolutely in $(\dom(V_z),\norm\cdot_{V_z})$. Hence it belongs to the closed linear span of $\{V_s y\mid s\in\IR\}$ in $(\dom(V_z),\norm\cdot_{V_z})$. Finally, since $D$ is invariant under $(V_s)_{s\in\IR}$, whenever $y\in D$, then the linear span of $\{V_s y\mid s\in\IR\}$ is contained in $D$. Altogether, we have proven that $x$ belongs to the closure of $D$ with respect to $\norm\cdot_{V_z}$.
\end{proof}

\subsection{Correspondences}

Let $M$, $N$ be von Neumann algebras. An \emph{$M$-$N$ correspondence} is a Hilbert space $\cH$ together with normal unital $\ast$-homomorphisms $\pi_l^\cH\colon M\to B(\cH)$ and $\pi_r^{\cH}\colon N^{\mathrm{op}}\to B(\cH)$ with commuting images. We use bimodule notation and write $x\xi y$ for $\pi_l^\cH(x)\pi_r^{\cH}(y^{\mathrm{op}})\xi$ whenever $x\in M$, $y\in N$ and $\xi\in\cH$. For example, if $(M,L^2(M),J,L^2_+(M))$ is a standard form, then $L^2(M)$ with the actions $x\xi y=x J y^\ast J \xi$ becomes an $M$-$M$ correspondence, called the \emph{trivial} correspondence.

We say that the pair $(\cH,\cJ)$ is a \emph{self-dual $M$-$M$ correspondence} if $\cH$ is an $M$-$M$ correspondence and $\cJ\colon\cH\to\cH$ is an anti-unitary involution such that $\cJ(x\xi y)=y^\ast (\cJ\xi)x^\ast$ for all $x,y\in M$ and $\xi\in \cH$. For example, if $(M,L^2(M),J,L^2_+(M))$ is a standard form, then $(L^2(M),J)$ is a self-dual $M$-$M$ correspondence.

Let $\phi$ be a faithful normal state on $M$. In \cite{Wir22b}, the second-named author introduced the concept of a \emph{Tomita correspondence} over $(M,\phi)$, which is a triple $(\cH,\cJ,(\mathcal U_t)_{t\in \IR})$ consisting of a self-dual $M$-$M$ correspondence $(\cH,\cJ)$ and a strongly continuous unitary group $(\mathcal U_t)_{t\in \mathbb R}$ such that $\mathcal U_t\cJ=\cJ \mathcal U_t$ for all $t\in \IR$ and $\mathcal U_t(x\xi y)=\sigma^\phi_t(x)(\mathcal U_t\xi)\sigma^\phi_t(y)$ for all $x,y\in M$, $\xi\in\cH$ and $t\in\IR$.

For example, if $(M,L^2(M),J,L^2_+(M))$ is a standard form and $\phi$ is a faithful normal state on $M$, then $(L^2(M),J,(\Delta_\phi^{it})_{t\in\IR})$ is a Tomita correspondence over $(M,\phi)$. More generally, if $M\subset\hat M$ is a unital inclusion of von Neumann algebras, $\hat\phi$ is a faithful normal state on $\hat M$ and $E\colon \hat M\to M$ is a $\hat\phi$-preserving conditional expectation, then $(L^2(\hat M),J_{\hat M},(\Delta_{\hat \phi}^{it})_{t\in\IR})$ is a Tomita correspondence over $(M,\hat\phi|_M)$. Moreover, using an operator-valued semicircular construction \cite{Shl99}, one can show (see \cite[Section 4.2]{Wir22b}, \cite[Section 5.2]{KW25}) that for every Tomita correspondence $(\cH,\cJ,(\mathcal U_t)_{t\in\IR})$ over $(M,\phi)$ there exists a von Neumann algebra $\hat M\supset M$, a faithful normal state $\hat\phi$ such that $\hat \phi|_M=\phi$ and a faithful normal conditional expectation $E\colon\hat M\to M$ such that $\cH\subset L^2(\hat M)$, $J_{\hat M}|_{\cH}=\cJ$ and $\Delta_{\hat\phi}^{it}|_{\cH}=\mathcal U_t$ for all $t\in\IR$.

\section{Characterization of bounded quantum Dirichlet forms}\label{sec:char_bdd}

In this section, we prove an integral representation for bounded quantum Dirichlet forms (\Cref{thm:int_rep_form}) before giving a full characterization of bounded quantum Dirichlet forms in terms of derivations (\Cref{thm:char_bdd}).

Throughout this section, $M$ denotes a von Neumann algebra and $\phi$ a faithful normal state on $M$.

We start with a simple consequence of the residue theorem, which is a variant of results commonly used in the proof of Tomita's theorem (see \cite[Section VI.1]{Tak03} for example).

\begin{lemma}\label{lem:appl_residue_thm}
    Let $S=\{z\in\mathbb C\mid -1/4<\Re z<1/4\}$. If $f\colon \overline S\to\IC$ is a bounded continuous function that is analytic on $S$, then
    \begin{equation*}
        f(0)=\int_{\mathbb R}(f(-1/4+it)+f(1/4+it))\frac{dt}{\cosh 2\pi t}.
    \end{equation*}
\end{lemma}
\begin{proof}
    For $z\in S\setminus\{0\}$ let $g(z)=\frac{f(z)}{\sin 2\pi z}$. The function is meromorphic on $S$ with a single pole with residue $f(0)/2\pi$ at $0$. Moreover, $\lim_{R\to\infty}\abs{g(s\pm iR)}=0$ uniformly in $s\in [-1/4,1/4]$.
    
    For $R,\epsilon>0$ let $\gamma_{R,\epsilon}$ be the positively oriented curve running along the boundary of $\{z\in \IC:\abs{\Re z}\leq \frac 1 4-\epsilon,\abs{\Im z}\leq R\}$. By the residue theorem,
    \begin{equation*}
        f(0)=-i\int_{\gamma_{R,\epsilon}}g(z)\,dz.
    \end{equation*}
    Since $g$ is continuous on $S\setminus \{0\}$, we have
    \begin{equation*}
        \lim_{\epsilon\to 0}\int_{\gamma_{R,\epsilon}}g(z)\,dz=\int_{\gamma_R}g(z),dz,
    \end{equation*}
    where $\gamma_R$ is the positively oriented curve running along the boundary of $\{z\in \IC:\abs{\Re z}\leq \frac 1 4,\abs{\Im z}\leq R\}$. Second, because $\lim_{R\to\infty}\abs{g(s\pm iR)}=0$ uniformly in $s\in [-1/4,1/4]$, we can take the limit $R\to\infty$ to get
    \begin{align*}
        f(0)
        &=-i\lim_{R\to\infty}\int_{\gamma_R}g(z)\,dz\\
        &=\int_\IR (g(1/4+it)-g(-1/4+it))\,dt\\
        &=\int_\IR (f(-1/4+it)+f(1/4+it))\frac{dt}{\cosh 2\pi t}.\qedhere
    \end{align*}
    
\end{proof}

As a first consequence, we obtain a new integral formula for the $\mathcal V$-transform.

\begin{proposition}
    If $T\in B(L^2(M))$ and $\xi,\eta\in L^2(M)$, then
    \begin{equation*}
        \langle \xi,\check T\eta\rangle=\int_\IR \langle \Delta_\phi^{it}\xi,T\Delta_\phi^{it}\eta\rangle\frac{2\,dt}{\cosh 2\pi t}.
    \end{equation*}
\end{proposition}
\begin{proof}
    First note that
    \begin{equation*}
        \int_\IR \abs{\langle \Delta_\phi^{it}\xi,T\Delta_\phi^{it}\eta\rangle}\frac{2\,dt}{\cosh 2\pi t}\leq \norm{T}\norm{\xi}\norm{\eta}\int_\IR\frac{2\,dt}{\cosh 2\pi t}=\norm{T}\norm{\xi}\norm{\eta}.
    \end{equation*}
    Hence the integral converges absolutely and there exists a unique $R\in B(L^2(M))$ such that $\langle \xi,R\eta\rangle=\int_\IR \langle \Delta_\phi^{it}\xi,T\Delta_\phi^{it}\eta\rangle\frac{2\,dt}{\cosh 2\pi t}$ for all $\xi,\eta\in L^2(M)$.
    
    If $\xi,\eta\in \dom(\Delta_\phi^{1/4})\cap \dom(\Delta_\phi^{-1/4})$, then
    \begin{align*}
        &\frac 1 2\langle \Delta_\phi^{1/4}\xi,R \Delta_\phi^{-1/4}\eta\rangle+\frac 1 2\langle \Delta_\phi^{-1/4}\xi,R \Delta_\phi^{1/4}\eta\rangle\\
        &\quad=\int_\IR(\langle \Delta_\phi^{1/4+it}\xi,T\Delta_\phi^{-1/4+it}\eta\rangle+\langle \Delta_\phi^{-1/4+it}\xi,T\Delta_\phi^{1/4+it}\eta\rangle)\frac{2\,dt}{\cosh 2\pi t}.
    \end{align*}
    For $z\in \IC$ with $\abs{\Re z}\leq \frac 1 4$ let $f(z)=\langle \Delta_\phi^{-\bar z}\xi,T\Delta_\phi^z \eta\rangle$. By \cite[Lemma VI.2.3]{Tak03}, the function $f$ is bounded and continuous on the closed strip $\{z\in \IC:\abs{\Re z}\leq \frac 1 4\}$ and analytic on its interior. It follows from \Cref{lem:appl_residue_thm} that
    \begin{align*}
        \int_\IR(\langle \Delta_\phi^{1/4+it}\xi,T\Delta_\phi^{-1/4+it}\eta\rangle+\langle \Delta_\phi^{-1/4+it}\xi,T\Delta_\phi^{1/4+it}\eta\rangle)\frac{2\,dt}{\cosh 2\pi t}=f(0)=\langle \xi,T\eta\rangle.
    \end{align*}
    Therefore, $R=\check T$ by \cite[Lemma 3.4]{VW23}.
\end{proof}

\begin{theorem}\label{thm:int_rep_form}
    Let $\mathcal E$ be a bounded quantum Dirichlet form on $L^2(M)$ with respect to $\Omega_\phi$. There exists a self-dual $M$-$M$ correspondence $(\tilde \cH,\tilde\cJ)$ and a vector $\xi\in\tilde\cH$ such that
        \begin{itemize}
            \item $\tilde\cJ\xi=\xi$,
            \item $\mathcal E(a)=\int_{\mathbb R}\norm{\pi_l(\Delta_\phi^{it}a)\xi-\xi\cdot J\pi_r(\Delta_\phi^{it}a)^\ast J}_{\mathcal H}^2\frac{2\,dt}{\cosh 2\pi t}$ for all $a\in M\Omega_\phi\cap M^\prime\Omega_\phi$.
        \end{itemize}
    Moreover, the map
    \begin{equation*}
        \tilde\delta\colon M\Omega_\phi\cap M^\prime \Omega_\phi\to \tilde\cH,\,a\mapsto \pi_l(a)\xi-\xi\cdot J\pi_r(a)^\ast J
    \end{equation*}
    is closable and
    \begin{equation*}
        \norm{\tilde\delta(a)}_{\tilde\cH}^2=\Re\mathcal E(\Delta_\phi^{1/4}a,\Delta_\phi^{-1/4}a)\leq \mathcal E(\Delta_\phi^{1/4}a)^{1/2}\mathcal E(\Delta_\phi^{-1/4}a)^{1/2}.
    \end{equation*}
\end{theorem}
\begin{proof}
    Let $L$ be the (bounded) positive self-adjoint operator such that $\mathcal E=Q_L$ and let $(\Phi_t)_{t\geq 0}$ be the associated uniformly continuous KMS-symmetric quantum Markov semigroup with generator $\mathcal L$. As a consequence of the Christensen--Evans theorem (see \cite[Proposition 4.1]{VW23} for the details), the operator $\mathcal L$ is the difference of two normal completely positive maps.

    Define
    \begin{equation*}
        \omega\colon M\odot M^\prime\to\IC,\,x\otimes y^\prime=-\frac 1 2\langle\Omega_M,\mathcal L(x)y^\prime\Omega_M\rangle
    \end{equation*}
    Since $\mathcal L$ is the difference of two normal completely positive maps, $\omega$ extends to a bounded linear functional on the binormal tensor product $M\binotimes M^\prime$ (see \cite[Section 2]{EL77a}), still denoted by $\omega$.
    
    Let 
    \begin{equation*}
        I_0=\left\{\sum_{j=1}^n x_j\otimes y_j^\prime\right\vert\left.x_1,\dots,x_n\in M,y_1,\dots,y_n\in M^\prime,\sum_{j=1}^n x_j y_j^\prime \Omega=0\right\}
    \end{equation*}
    and $I=\overline I_0\subset M\binotimes M^\prime$, $B=I\cap I^\ast$.
    
    Note that $I_0\subset M\odot M^\prime$ is a left ideal and hence $I\subset M\binotimes M^\prime$ is a closed left ideal. Moreover, $\omega|_{B}$ is positive since
    \begin{align*}
        \omega\left(\sum_{j,k=1}^n x_j^\ast x_k\otimes (y_j^\prime)^\ast y_k^\prime\right)&=\lim_{t\to 0}\frac 1 t\sum_{j,k=1}^n \langle y_j^\prime\Omega,\Phi_t(x_j^\ast x_k)y_k^\prime\Omega\rangle\geq 0
    \end{align*}
    for $\sum_{j=1}^n x_j\otimes y_j^\prime\in I_0$.

    Let
    \begin{equation*}
        \alpha\colon M\odot M^\prime\to M\odot M^\prime,\,x\otimes y^\prime\mapsto J y^\prime J\otimes J x J.
    \end{equation*}
    The map $\alpha$ is a anti-linear, unital, multiplicative and self-adjoint ($\alpha(v)^\ast=\alpha(v^\ast)$) and satisfies $\alpha^2=\mathrm{id}$. Note that if $\psi\colon M\odot M^\prime\to\IC$ is a binormal state, then $v\mapsto \overline{\psi(\alpha(v))}$ is again a binormal state. From this fact it follows easily that $\alpha$ extends to an isometric involution on $M\binotimes M^\prime$.
    
    Moreover, as a consequence of KMS symmetry,
    \begin{align*}
        \omega(\alpha(x\otimes y^\prime))
        &=\omega(J y^\prime J\otimes J xJ)\\
        &=-\frac 1 2\langle \Omega_\phi,\mathcal L(J y^\prime J)J xJ\Omega_\phi\rangle\\
        &=-\frac 1 2\langle J\Omega_\phi,J y^\prime \mathcal L(x)\Omega_\phi\rangle\\
        &=-\frac 1 2 \langle y^\prime \mathcal L(x)\Omega_\phi,\Omega\phi\rangle\\
        &=\overline{\omega(x\otimes y^\prime)}.
    \end{align*}

    The functional $\omega|_B$ can be (uniquely) extended to a positive linear functional on $B+\IC 1$ and then by Hahn--Banach to a positive linear functional $\tilde\omega$ on $M\binotimes M^\prime$. A direct computation shows $\tilde\omega(\alpha(v))=\overline{\omega(v)}$ for $v\in B+\IC 1$. Thus, otherwise replacing $\tilde \omega$ by the functional $\frac  12(\tilde\omega+\overline{\tilde\omega(\alpha(\cdot))})$, we may assume that $\tilde \omega\circ\alpha=\overline{\tilde \omega(\cdot)}$.

    Let $(\rho,\tilde\cH,\xi)$ be the cyclic representation induced by $\tilde\omega$. The Hilbert space $\tilde\cH$ becomes an $M$-$M$ correspondence with the left and right action defined by $x\eta y=\rho(x\otimes Jy^\ast J)\eta$.
    
    For $v,w\in M\binotimes M^\prime$ we have
    \begin{equation*}
        \langle \alpha(v)\xi,\alpha(w)\xi\rangle=\tilde\omega(\alpha(v^\ast w))=\overline{\tilde\omega(v^\ast w)}=\overline{\langle v\xi,w\xi\rangle}.
    \end{equation*}
    Thus there exists a unique anti-unitary involution $\tilde\cJ$ on $\tilde\cH$ such that $\tilde\cJ(v\xi)=\alpha(v)\xi$ for all $v\in M\binotimes M^\prime$. A direct computation shows $\tilde\cJ(x\eta y)=y^\ast(\tilde\cJ \eta)x^\ast$ for $x,y\in M$ and $\eta\in (M\binotimes M^\prime)\xi$, and this identity extends to arbitrary $\eta\in \tilde\cH$ by continuity. Furthermore, $\tilde\cJ\xi=\alpha(1)\xi=\xi$.

    If $a\in M\Omega_\phi\cap M^\prime \Omega_\phi$, then $\pi_l(a)\otimes 1-1\otimes \pi_r(a)\in I$, hence
    \begin{align*}
        \norm{\pi_l(a)\xi-\xi\cdot J\pi_r(a)^\ast J}_{\tilde\cH}^2&=\omega(\abs{\pi_l(a)\otimes 1-1\otimes \pi_r(a)}^2)\\
        &=\frac  1 2\langle \Omega_\phi,\mathcal L(\pi_l(a))^\ast \pi_r(a)\Omega_\phi\rangle+\frac1  2\langle \Omega_\phi,\mathcal L(\pi_l(a))\pi_r(a)^\ast\Omega_\phi\rangle\\
        &=\frac 1 2\langle \Delta_\phi^{1/4}\mathcal L(\pi_l(a))\Omega_\phi,\Delta_\phi^{-1/4}a\rangle+\frac 1 2\langle \Delta_\phi^{-1/4}a,\Delta_\phi^{1/4}\mathcal L(\pi_l(a))\Omega_\phi\rangle\\
        &=\frac 1 2\mathcal E(\Delta_\phi^{1/4}a,\Delta_\phi^{-1/4}a)+\frac 1 2\mathcal E(\Delta_\phi^{-1/4}a,\Delta_\phi^{1/4}a).
    \end{align*}
    Note that the manipulation with the modular operator are justified as $M\Omega_\phi\subset \dom(\Delta_\phi^{1/4})$ and $M^\prime\Omega_\phi\subset\dom(\Delta_\phi^{-1/4})$. In particular, 
    \begin{equation*}
        M\Omega_\phi\cap M^\prime\Omega_\phi\subset \dom(\Delta_\phi^{1/4}L\Delta_\phi^{-1/4})\cap \dom(\Delta_\phi^{-1/4}L\Delta_\phi^{1/4}),
    \end{equation*}
    which implies that the quadratic form $Q$ with $\dom(Q)=M\Omega_\phi\cap M^\prime\Omega_\phi$ defined by
    \begin{equation*}
        Q(a)= \frac 1 2\mathcal E(\Delta_\phi^{1/4}a,\Delta_\phi^{-1/4}a)+\frac 1 2\mathcal E(\Delta_\phi^{-1/4}a,\Delta_\phi^{1/4}a)=\langle a,\frac 1 2(\Delta_\phi^{1/4}L\Delta_\phi^{-1/4}+\Delta_\phi^{-1/4}L \Delta_\phi^{1/4})a\rangle
    \end{equation*}
    is closable \cite[Lemma 10.16]{Sch12}. Thus $\tilde\delta$ as defined in the statement of the theorem is closable. Moreover, the bound $\norm{\tilde\delta(a)}_{\tilde\cH}^2\leq \mathcal E(\Delta_\phi^{1/4}a)^{1/2}\mathcal E(\Delta_\phi^{-1/4}a)^{1/2}$ follows from the equation above by Cauchy--Schwarz.
    
    Let $S=\{z\in\IC\mid -1/4<\Re z<1/4\}$ and
    \begin{align*}
        f\colon\overline S\to\IC,\,z\mapsto \mathcal E(\Delta_\phi^{-\bar z}a,\Delta_\phi^{z} a).
    \end{align*}
    Since $a\in\dom(\Delta_\phi^{1/2})\cap\dom(\Delta_\phi^{-1/2})$ and $\mathcal E$ is bounded, the function $f$ is continuous and bounded on $\overline S$ and analytic on $S$ by \cite[Lemma VI.2.3]{Tak03}. It satisfies $f(0)=\langle a,La\rangle=\mathcal E(a)$ and
    \begin{align*}
        f(-1/4+it)+f(1/4+it)&=\mathcal E(\Delta_\phi^{1/4+it}a,\Delta_\phi^{-1/4+it}a)+\mathcal E(\Delta_\phi^{-1/4+it}a,\Delta_\phi^{1/4+it}a)\\
        &=2\norm{\pi_l(\Delta_\phi^{it}a)\xi-\xi\cdot J\pi_r(\Delta_\phi^{it}a)^\ast J}^2
    \end{align*}
    by the previous computation.

    Hence we can apply \Cref{lem:appl_residue_thm} to get
    \begin{align*}
        \mathcal E(a)&=f(0)\\
        &=\int_{\mathbb R}(f(-1/4+it)+f(1/4+it))\frac{dt}{\cosh 2\pi t}\\
        &=\int_{\mathbb R}\norm{\pi_l(\Delta_\phi^{it}a)\xi-\xi\cdot J\pi_r(\Delta_\phi^{it}a)^\ast J}^2\frac{2\,dt}{\cosh 2\pi t}.\qedhere
    \end{align*}
\end{proof}

\begin{remark}
    This result and its proof are very close to \cite[Theorem 2.3]{Ko07}. The differences are that for one we show that the derivation $\delta$ is in fact inner in a suitable sense, and, more importantly, that we focus on the Hilbert space implementation of the derivation instead of its implementation on the level of the underlying von Neumann algebra. This formulation lends itself much better to the application to unbounded quantum Dirichlet forms later (compare with \cite{Ko09}, where Ko also obtained a representation result for unbounded quantum Dirichlet forms, but only under strong assumptions on the generator).
\end{remark}

\begin{remark}
    For specific bounded quantum Dirichlet forms, for which the generator of the associated quantum Markov semigroup takes the form $\cL=1-\Psi$ with a KMS-symmetric ucp map $\Psi$, a form of this representation theorem was proven in \cite{Wir25} and used to establish the triangle inequality for a quantum Wasserstein distance. In this special case, the vector $\xi\in \tilde\cH$ is always a bounded vector so that the quantum Dirichlet form takes the form
    \begin{equation*}
        \mathcal E(a)=\int_{\IR}\norm{(L_\phi(\xi)-R_\phi(\xi))\Delta_\phi^{it}a}_{\tilde\cH}^2\frac{2\,dt}{\cosh 2\pi t},\quad a\in L^2(M).
    \end{equation*}
\end{remark}

Conversely, if $\tilde\cH$, $\tilde\cJ$, $\xi$ are as in the statement of the theorem and the quadratic form $\mathcal E_0$ with domain $M\Omega_\phi\cap M^\prime\Omega_\phi$ given by
\begin{equation*}
    \mathcal E_0(a)=\int_\IR\norm{\pi_l(\Delta_\phi^{it}a)\xi-\xi\cdot J\pi_r(\Delta_\phi^{it}a)^\ast J}^2\frac{2\,dt}{\cosh 2\pi t},\quad a\in M\Omega_\phi\cap M^\prime\Omega_\phi
\end{equation*}
is bounded, then its closure is a quantum Dirichlet form. This will follow from \Cref{thm:char_bdd} below. However, in the light of the results of Park \cite{Par00}, the density $\frac{2}{\cosh 2\pi t}$ is not the only possible choice, and as far as we can say, it is not a canonical choice. Moreover, we do not have any direct criterion that allows to check boundedness of $\mathcal E_0$ in terms of $(\tilde\cH,\tilde\cJ,\xi)$. For this reason we chose to present the characterization of quantum Dirichlet forms in a slightly different, more abstract way.

First, we need a result on the translation group on the weighted $L^2$ space with weight $\frac{2}{\cosh 2\pi t}$. For the rest of the section, let $\mu$ be the measure on $\IR$ with density $\frac{d\mu}{dt}=\frac{2}{\cosh 2\pi t}$.

\begin{proposition}\label{prop:trans_weighted_L2}
    Let $H$ be a Hilbert space. The family $(\mathcal V_s)_{s\in \IR}$ given by 
    \begin{equation*}
        \mathcal V_s\colon L^2(\IR,\mu;H)\to L^2(\IR,\mu;H),\,(\mathcal V_s \xi)(t)=\xi(s+t)
    \end{equation*}
    for $s\in\IR$ is a strongly continuous group, $\norm{\cV_s}\leq e^{\pi\abs{s}}$ for all $s\in\IR$ and $\cV_{-i/2}$ is accretive.
\end{proposition}
\begin{proof}
    Since $\frac{\cosh 2\pi t}{\cosh 2\pi(t-s)}\in [e^{-2\pi\abs{s}},e^{2\pi\abs{s}}]$ for all $s,t\in\IR$, we have
    \begin{align*}
        \norm{\cV_s\xi}^2=\int_\IR \norm{\xi(t)}^2\frac{\cosh 2\pi t}{\cosh{2\pi(t-s)}}\frac{2\,dt}{\cosh 2\pi t}\leq e^{2\pi\abs{s}}\norm{\xi}^2
    \end{align*}
    for all $\xi\in L^2(\IR,\mu;H)$. For strong continuity, it is therefore enough to prove $\norm{\cV_s \xi-\xi}_2\to 0$ for $\xi\in C_c(\IR;H)$, which is obvious.

    It remains to show that $\cV_{-i/2}$ is accretive. Let $S=\{z\in \IC\mid -\frac 1 2\leq \Im z\leq 0\}$. By definition, $\dom(\cV_{-i/2})\subset \dom(\cV_z)$ for all $z\in S$, and if $\xi\in\dom(\cV_{-i/2})$, then $S\to L^2(\IR,\mu;\cH)$, $z\mapsto \cV_z \xi$ is continuous and analytic on the interior of $S$.

    Let $\phi\in C_c^\infty(\IR)$ be positive with $\int_\IR\phi(t)\,dt=1$ and let $\phi_\epsilon(t)=\epsilon^{-1}\phi(t/\epsilon)$ for $\epsilon>0$. Define
    \begin{equation*}
        h_\epsilon\colon \{z\in \IC\mid -1/4\leq \Re z\leq 1/4\}\to \IC,\,z\mapsto\int_{\IR}\phi_\epsilon(t)\langle \cV_{i\bar z-i/4}\xi(t),\cV_{-iz-i/4}\xi(t)\rangle_\cH\,dt.
    \end{equation*}
    Note that $L^1(\IR,\mu)\subset L^1_{\mathrm{loc}}(\IR)$ so that the integral is well-defined. For $s\in\IR$ let $V(s)=-\log\frac{2}{\cosh 2\pi s}$. By \Cref{lem:appl_residue_thm}, we have 
    \begin{align*}
        \int_{\IR}\frac{h_\epsilon(-1/4+is)+h_\epsilon(1/4+is)}{2}e^{-V(s)}\,ds=h_\epsilon(0)=\int_\IR\phi_\epsilon(t)\norm{\cV_{-i/4}\xi(t)}_\cH^2\,dt\geq 0.
    \end{align*}
    On the other hand,
    \begin{align*}
        \int_\IR h_\epsilon(-1/4+is)e^{-V(s)}\,ds
        &=\int_\IR\int_\IR\phi_\epsilon(t)\langle \cV_{s-i/2}\xi(t),\cV_{s}\xi(t)\rangle_\cH e^{-V(s)}\,dt\,ds\\
        &=\int_\IR\int_\IR\phi_\epsilon(t)\langle \cV_{-i/2}\xi(s+t),\xi(s+t)\rangle_\cH e^{-V(s)}\,dt\,ds\\
        &=\int_\IR\int_\IR \langle \cV_{-i/2}\xi(s),\xi(s)\rangle_\cH e^{-V(s)}\int_{\IR}\phi_\epsilon(t) e^{V(s)-V(s-t)}\,dt\,ds.\\
    \end{align*}
    If $\epsilon_n\to 0$, then $\int_{\IR}\phi_{\epsilon_n}(t) e^{V(s)-V(s-t)}\,dt\to 1$ for every $s\in\IR$ by the properties of the mollifier. Moreover, if $\operatorname{supp}\phi\subset [-R,R]$, then $\int_{\IR}\phi_\epsilon(t) e^{V(s)-V(s-t)}\,dt\leq \int_{\IR}\phi_\epsilon(t)e^{2\pi \abs{t}}\,dt\leq e^{2\pi R\epsilon}$. Therefore,
    \begin{align*}
        \int_\IR h_{\epsilon_n}(-1/4+is)e^{-V(s)}\,ds\to \int_\IR \langle \cV_{-i/2}\xi(s),\xi(s)\rangle_\cH e^{-V(s)}\,ds=\langle \cV_{-i/2}\xi,\xi\rangle_\cH
    \end{align*}
    by the dominated convergence theorem.

    An analogous argument shows
    \begin{equation*}
        \int_\IR h_{\epsilon_n}(1/4+is)e^{-V(s)}\,ds\to\langle\xi,\cV_{-i/2}\xi\rangle_\cH.
    \end{equation*}
    Altogether we have shown $\Re\langle \xi,\cV_{-i/2}\xi\rangle\geq 0$.
\end{proof}

\begin{remark}
    It follows from Monniaux's theorem \cite[Theorem 4.3]{Mon99} that $\mathcal V_{-i/2}$ is closed and $\cV_{-i/2}+\lambda$ is surjective for all $\lambda>0$. Thus the operator $\cV_{-i/2}$ is in fact $m$-accretive.
\end{remark}

\begin{lemma}\label{lem:dense_real_part}
The set $\{x^\ast \Omega_\phi x-y\Omega_\phi y^\ast \mid x,y\in M,\,xy=0\}$ is dense in $\ker (1-J)$.
\end{lemma}
\begin{proof}
    If $a\in \ker(1-J)$, then the support projections $s(a_+)$ and $s(a_-)$ are orthogonal. By \cite[Theorem 1.1]{Haa75} and the uniqueness of standard forms, there exist sequences $(x_n)$, $(y_n)$ in $M$ such that $x_n^\ast \Omega x_n\to a_+$ and $y_n \Omega y_n^\ast\to a_-$. We have $s(a_+)x_n^\ast \Omega x_ns(a_+)\to a_+$, $s(a_-)y_n \Omega y_n^\ast s(a_-)\to a_-$ and $x_n s(a_+)s(a_-)y_n=0$.
\end{proof}

\begin{theorem}\label{thm:char_bdd}
    A linear operator $L$ on $L^2(M)$ is the generator of a quantum Dirichlet form if and only if there exist
    \begin{itemize}
        \item an $M$-$M$ correspondence $\cH$,
        \item an anti-unitary involution $\mathcal J\colon\cH\to\cH$,
        \item a strongly continuous one-parameter group $(\cV_s)_{s\in\IR}$ on $\mathcal H$ with $\cV_{-i/2}$ accretive,
        \item a bounded operator $\delta\colon L^2(M)\to\mathcal H$
    \end{itemize}
    satisfying
    \begin{enumerate}[(i)]
        \item $\cH=\overline{\operatorname{span}}\{x\delta(a)y\mid a\in L^2(M),x,y\in M\}$,
        \item $\mathcal J(x\eta y)=y^\ast (\mathcal J\eta)x^\ast$, $\cV_s(x\eta y)=\sigma^\phi_s(x)(\cV_s\eta)\sigma^\phi_s(y)$ for all $x,y\in M$, $\eta\in\cH$, $s\in\IR$,
        \item $\cV_s\mathcal J=\mathcal J \cV_s$ for all $s\in\IR$,
        \item $\delta\circ J=\mathcal J\circ\delta$, $\delta\circ \Delta_\phi^{is}=\cV_s\circ\delta$,
        \item $\delta(x\Omega_\phi y)=x\delta(\Omega_\phi y)+\delta(x\Omega_\phi)y$ for all $x,y\in M$
    \end{enumerate}
    such that
    \begin{equation*}
        L=\delta^\ast\delta.
    \end{equation*}
    Moreover, if $(\cH^\prime,\cJ^\prime,(\mathcal V_s^\prime)_{s\in\IR},\delta^\prime)$ is a quadruple consisting of a self-dual correspondence $(\cH^\prime,\cJ^\prime)$, a strongly continuous group $(\cV^\prime_s)_{s\in\IR}$ on $\cH^\prime$ and a bounded operator $\delta^\prime\colon L^2(M)\to \cH^\prime$ satisfying (i)--(v) above, then there exists a unique unitary bimodule map $\Theta\colon \cH\to \cH^\prime$ such that $\Theta\circ \cJ=\cJ^\prime\circ\Theta$, $\Theta\circ \cV_s=\cV_s^\prime\circ \Theta$ for all $s\in \IR$ and $\Theta\circ \delta=\delta^\prime$.
    
    The operator $L$ commutes strongly with $\Delta_\phi$ if and only if $\cV_{-i/2}$ is a positive self-adjoint operator.
\end{theorem}
\begin{proof}
    First assume that $L$ generates a bounded conservative completely Dirichlet form. By \Cref{thm:int_rep_form} there exists an $M$-$M$ correspondence $\tilde\cH$, an anti-unitary involution $\tilde\cJ\colon \tilde\cH\to \tilde\cH$ and a vector $\xi\in \tilde\cH$ such that $\tilde \cJ(x\xi y)=y^\ast \xi x^\ast$ for all $x,y\in M$ and
    \begin{equation*}
        \mathcal E(a)=\int_{\IR}\norm{\pi_l(\Delta_\phi^{it}a)\xi-\xi\cdot J\pi_r(\Delta_\phi^{it}a)^\ast J}_{\tilde\cH}^2\frac{2\,dt}{\cosh 2\pi t},\quad a\in M\Omega\cap M^\prime\Omega.
    \end{equation*}
    We may assume without loss of generality that $\tilde \cH$ is the closed bimodule generated by $\xi$.

    On $L^2(\IR,\mu;\tilde\cH)$ define a bimodule structure by $(x\eta y)(t)=\sigma^\phi_t(x)\eta(t)\sigma^\phi_t(y)$ for $x,y\in M$ and $\eta\in L^2(\IR,\mu;\tilde\cH)$. Let 
    \begin{equation*}
        \delta\colon M\Omega\cap M^\prime\Omega\to L^2(\IR,\mu;\tilde\cH),\,\delta(a)(t)=\pi_l(\Delta_\phi^{it}a)\xi-\xi\cdot J\pi_r(\Delta_\phi^{it}a)^\ast J
    \end{equation*}
    and let $\cH$ be the closed bimodule generated by the range of $\delta$. Since
    \begin{align*}
        \norm{\delta(a)}_\cH^2=\int_{\IR}\norm{\pi_l(\Delta_\phi^{it}a)\xi-\xi\cdot J\pi_r(\Delta_\phi^{it}a)^\ast J}_{\tilde\cH}^2\frac{2\,dt}{\cosh 2\pi t}=\mathcal E(a)=\langle a,L(a)\rangle_2
    \end{align*}
    for $a\in M\Omega\cap M^\prime\Omega$, the operator $\delta$ extends to a bounded linear operator from $L^2(M)$ to $\cH$, still denoted by $\delta$, such that $\delta^\ast \delta=L$.

    Define $(\cJ\eta)(t)=\tilde\cJ\eta(t)$ and $(\cV_s\eta)(t)=\eta(s+t)$ for $\eta\in \cH$. The claimed properties of the quadruple $(\cH,\cJ,(\cV_s)_{s\in\IR},\delta)$ are now easy to check except for accretivity of $\cV_{-i/2}$, which was proven in \Cref{prop:trans_weighted_L2}.

    For uniqueness, we have already shown in \cite[Theorem 4.7]{VW23} that the triple $(\cH,\cJ,\delta)$ is essentially uniquely determined by $L$. Moreover, the properties (ii) and (iv) determine $(\cV_s)_{s\in\IR}$ uniquely on the dense subset $
    \operatorname{span}\{x\delta(a)y\mid a\in L^2(M),x,y\in M\}$ of $\cH$.

    From uniqueness it follows that if $L$ commutes strongly with $\Delta_\phi$, then $(\cV_s)_{s\in \IR}$ is a strongly continuous unitary group and $\cV_{-i/2}$ is positive self-adjoint by Stone's theorem. Conversely, if $\cV_{-i/2}$ is positive self-adjoint (and injective since $\cV_{i/2}\cV_{-i/2}=1$ on $\dom(\cV_{-i/2})$), then $\cV_s=\cV_{-i/2}^{2is}$ is unitary and thus
    \begin{equation*}
        \Delta_\phi^{-is}L\Delta_\phi^{is}=\Delta_\phi^{-is}\delta^\ast\delta \Delta_\phi^{is}=\delta^\ast \cV_s^\ast\cV_{s}\delta=\delta^\ast \delta=L.
    \end{equation*}
    Now assume that $(\cH,\cJ,(\cV_s)_{s\in\IR},\delta)$ is a quadruple with the properties from the theorem and let $L=\delta^\ast \delta$. Clearly, $L$ is a positive self-adjoint operator. Moreover, the properties of $\cJ$ and $\delta$ imply $LJ=JL$ and $L\Omega_\phi=0$. Note that if $x,y\in \dom(\sigma^\phi_{z})$ and $a\in \dom(\Delta_\phi^{iz})$, then $x\delta(a)y\in \dom(\cV_{z})$ and $\cV_z(x\delta(a)y)=\sigma^\phi_z(x)\delta(\Delta_\phi^{iz}a)\sigma^\phi_z(y)$. Indeed, for $z\in \IR$ this is immediate from the properties of $(\cV_s)_{s\in \IR}$ and $\delta$ and for general $z\in\IC$ it follows them from the identity theorem for analytic functions.
    
    Let $\mathcal E$ be the closed quadratic form associated with $L$. It remains to show that $\mathcal E(a_+,a_-)\leq 0$ for all $a\in \ker(1-J)$. In fact, since taking positive and negative parts is continuous, it suffices to show this inequality for $a$ in the subset $\{x^\ast\Omega_\phi x-y\Omega_\phi y^\ast\mid x,y\in M,\,xy=0\}$ of $\ker (1-J)$, which is dense by \Cref{lem:dense_real_part}.

    Let $a=x^\ast \Omega_\phi x-y\Omega_\phi y^\ast$ with $x,y\in M$ such that $xy=0$. By Haagerup's approximation technique \cite[Lemma 1.3]{Haa75}, there exist bounded sequences $(x_n)$, $(y_n)$ of entire analytic elements for $\sigma^\phi$ such that $x_n\to x$ and $y_n\to y$ in strong$^\ast$ topology. In particular, $x_ny_n\to 0$ in strong$^\ast$ topology.
    
    We have
    \begin{align*}
        \mathcal E(a_+,a_-)&=\langle \delta(x^\ast \Omega_\phi x),\delta(y\Omega_\phi y^\ast)\rangle_\cH\\
        &=\langle x^\ast \delta(\Omega_\phi x)+\delta(x^\ast \Omega_\phi)x,\delta(y\Omega_\phi)y^\ast+y\delta(\Omega_\phi y^\ast)\rangle\\
        &=\langle \delta(\Omega_\phi x)y,x\delta(y\Omega_\phi)\rangle+\langle y^\ast \delta(x^\ast\Omega_\phi),\delta(\Omega_\phi y^\ast)x^\ast\rangle\\
        &=\langle \delta(\Omega_\phi x)y,x\delta(y\Omega_\phi)\rangle+\langle \cJ(\delta(\Omega_\phi x)y),\cJ(x\delta(y\Omega_\phi))\rangle\\
        &=2\Re\langle \delta(\Omega_\phi x)y,x\delta(y\Omega_\phi)\rangle
        \end{align*}
        Now we use that $\delta(\Omega_\phi x_n)y_n\to \delta(\Omega_\phi x)y$ and $x_n\delta(y_n\Omega_\phi)\to x\delta(y\Omega_\phi)$ to obtain
        \begin{align*}
        \mathcal E(a_+,a_-)
        &=2\lim_{n\to\infty}\Re\langle \delta(\Omega_\phi x_n)y_n,x_n\delta(y_n\Omega_\phi)\rangle\\
        &=2\lim_{n\to\infty}\Re\langle \delta(\Omega_\phi x_n y_n)-\sigma^\phi_{-i/2}(x_n)\delta(\Omega_\phi y_n),x_n\delta(y_n\Omega_\phi)\rangle\\
        &=-2\lim_{n\to\infty}\Re\langle \cV_{-i/2}(x_n\delta(y_n\Omega_\phi)),x_n \delta(y_n\Omega_\phi)\rangle\\
        &\leq 0.\qedhere
    \end{align*}
\end{proof}

\begin{remark}\label{rem:norm_V_s}
    It follows from \Cref{prop:trans_weighted_L2} that $\norm{\cV_s}\leq e^{\pi\abs s}$ for all $s\in \IR$.
\end{remark}

\section{\texorpdfstring{Characterization of bounded quantum Dirichlet forms on $B(H)$}{Characterization of bounded quantum Dirichlet forms on B(H)}}\label{sec:type_I}

In this section, we specialize the characterization of bounded quantum Dirichlet forms to the case when the underlying von Neumann algebra is a type $\mathrm{I}$ factor.

Throughout this section, let $H$ be a Hilbert space. We write $S^2(H)$ for the space of Hilbert--Schmidt operators and recall that if $\pi_L$ denotes the representation of $B(H)$ on $S^2(H)$ by left multiplication, $JA=A^\ast$ and $S^2_+(H)$ denotes the set of positive Hilbert--Schmidt operators on $H$, then $(\pi_L(B(H)),S^2(H),J, S^2_+(H))$ is a standard form.

Letting $\Tr$ be the standard trace on $B(H)$, there exists a unique element $\rho\in S^1(H)$, the set of trace-class operators on $H$, such that $\Tr(x\rho)=\phi(x)$ for all $x\in M$. Let $p_n=\chi_{[1/n,n]}(\rho)$, so $\rho_n=p_n\rho p_n$ is invertible in $p_nB(H)p_n\cong B(p_nH)$.

\begin{lemma}
    Let $(\Phi_t)_{t\geq 0}$ be a uniformly continuous quantum Markov semigroup on $B(H)$ that is KMS-symmetric with respect to the state $\phi$. If $\cL$ is the generator of $(\Phi_t)_{t\geq 0}$ and $\cL_n=p_n\cL(p_n\cdot p_n)p_n$, then $\cL_n$ generates a quantum Markov semigroup on $B(p_n H)$ that is KMS-symmetric with respect to $\phi(p_n\cdot p_n)$.
\end{lemma}
\begin{proof}
    Clearly $\cL_n(1)=0$ and $\cL_n(x^\ast)=\cL_n(x)^\ast$ for $x\in B(p_n H)$. Moreover, if $x_1,\dots,x_m$, $y_1,\dots,y_m\in B(p_n H)$ such that $0=\sum_{j=1}^m x_j y_j=\sum_{j=1}^m x_j p_n y_j $, then 
    \begin{equation*}
        \sum_{j,k=1}^my_j^\ast p_n\cL(p_n x_j^\ast x_k p_n)p_n y_k\leq 0
    \end{equation*}
    since $\cL$ is conditionally completely negative. Therefore, $\cL_n$ generates a quantum Markov semigroup on $B(p_n H)$.

    Additionally, for any $x,y\in B(p_n H)$ we have
\begin{align*}
    \Tr(\rho_n^{1/2}x^*\rho_n^{1/2}\mathcal{L}_n(y))&=\Tr(\rho_n^{1/2}x^*\rho_n^{1/2}p_n\mathcal{L}(y)p_n)\\
    &=\Tr(\rho^{1/2}p_nx^*p_n\rho^{1/2}\cL(y))\\
    &=\Tr(\rho^{1/2}\cL(p_nx^*p_n)\rho^{1/2}y)\\
    &=\Tr(\rho_n^{1/2}\mathcal{L}_n(x^*)\rho_n^{1/2}y),
\end{align*}
where we use that $p_n$ and $\rho$ commute, $\cL$ is KMS-symmetric and $p_nzp_n=z$ for any $z\in B(p_n H)$. Consequently, $\cL_n$ is a KMS-symmetric generator on $B(p_n H)$. 
\end{proof}

Generators of uniformly continuous quantum Markov semigroups on general von Neumann algebras, without any symmetry assumptions, have been fully characterized by Christensen and Evans \cite{CE79} in the following form (see also \cite{Lin76} in the case of an injective von Neumann algebra).

\begin{theorem}[Christensen-Evans theorem]
    Let $M$ be a von Neumann algebra. A linear map $\cL\colon M\to M$ is the generator of a uniformly continuous quantum Markov semigroup if and only if there exist an element $k$ of $M$ and a normal completely positive map $\Psi$ on $M$ such that
    \begin{equation*}
        \cL(x)=k^*x+xk-\Psi(x)
    \end{equation*}
    for all $x\in M$.
\end{theorem}

If one additionally assumes that $\cL$ satisfies a symmetry condition, say GNS or KMS symmetry, it is natural to ask whether the two parts $x\mapsto k^\ast x+x k$ and $\Psi$ in the Christensen--Evans theorem can be chosen to be symmetric as well.

In the GNS-symmetric case, this was proven in \cite{Wir24a}. If $M=B(H)$, the KMS-symmetric case was settled by Fagnola and Umanità in \cite[Theorem 18]{FU10} (see also Remark 6). Here we give a new proof of their result with the additional benefit of a universal bound on the norms of $k$ and $\Psi$ in terms of the norm of $\cL$.

\begin{proposition}\label{thm:CE_KMS}
    If $\cL$ is a KMS-symmetric generator of a uniformly continuous quantum Markov semigroup, then there exist an element $k\in M$ with $\norm{k}\leq 4\norm{\cL}$ and a normal KMS-symmetric completely positive map $\Psi$ on $M$ such that
    $\norm{\Psi}\leq 2\norm{\cL}$, $k^*+k=\Psi(1)$ and 
    \begin{equation*}
        \cL(x)=k^*x+xk-\Psi(x).
    \end{equation*}
\end{proposition}

\begin{proof}
    Let $\tilde{\Psi}$ and $\tilde{k}$ be as in the Christensen-Evans theorem such that
    \begin{equation*}
        \cL(x)=\tilde{k}^*x+x\tilde k-\tilde{\Psi}(x)
    \end{equation*}
    for all $x\in M$. Let $\tilde{\Psi}(\cdot)=\sum_{j\in I}V_j^*\cdot V_j$ be the Kraus decomposition of $\tilde{\Psi}$ and let $\eta\in H$ be an eigenvector of $\rho$ with eigenvalue $\lambda$. Let $n\in \mathbb{N}$ be arbitrary such that $\lambda\in [1/n,n]$. In this case, $\cL_n$ is given by
    \begin{equation*}
        \cL_n(x)=p_nk^*p_nx+xp_nkp_n-\sum_{j\in I}p_nV_j^*p_n x p_n V_j p_n
    \end{equation*}
    for all $x\in M_n$. Since $\cL_n$ is KMS-symmetric, we have
    \begin{align*}
        \Tr\left(\rho_n^{1/2}\cL_n(x^*)\rho_n^{1/2}y\right)=&\Tr\left(\rho_n^{1/2}x^*\rho_n^{1/2}\cL_n(y)\right)\\
        =&\Tr\left(\rho_n^{1/2}x^*\rho_n^{1/2}(p_nk^*p_n y+yp_nkp_n-\sum_{j\in I}p_nV_j^*p_nyp_nV_jp_n)\right)\\
        =&\Tr\Bigg(\rho_n^{1/2}(\sigma^\phi_{i/2}(p_nkp_n)x^*+x^*\sigma^\phi_{-i/2}(p_nk^*p_n)\\
        &-\sum_{j\in I}\sigma^\phi_{i/2}(p_nV_jp_n)x^*\sigma^\phi_{-i/2}(p_nV_j^*p_n))\rho_n^{1/2}y\Bigg),
    \end{align*}
    using the fact that $\sigma^\phi_{-i/2}$ and $\sigma^\phi_{i/2}$ are bounded on $B(p_n H)$. Now we can define
    \begin{equation*}
        \Psi'(x)=\frac{1}{2}\sum_{j\in I}\bigg(p_nV_j^*p_nxp_nV_jp_n+\sigma^\phi_{i/2}(p_nV_jp_n)x\sigma^\phi_{-i/2}(p_nV_j^*p_n)\bigg)
    \end{equation*}
    for $x\in B(p_n H)$ and 
    \begin{equation*}
        k'=\frac{1}{2}(p_nkp_n+\sigma^\phi_{-i/2}(p_nk^*p_n)).
    \end{equation*}
    These definitions give
    \begin{equation*}
        \cL_n(x)=(k')^*x+xk'-\Psi'(x),\quad x\in B(p_n H),
    \end{equation*}
    and by construction $\Psi'$ is KMS-symmetric. We also see that $k'=\sigma_{-i/2}((k')^*)$. Together with the relation that $k'+(k')^*=\Psi'(1)$, this gives
    \begin{equation*}
        (k')^*=(\id+\sigma^\phi_{-i/2})^{-1}(\Psi'(1))\text{ and }k'=(\id+\sigma_{i/2})^{-1}(\Psi'(1)).
    \end{equation*}
    We have now obtained a KMS-symmetric $\Psi'$, but we are yet to establish a bound on its norm. For this, we need to modify $\Psi'$ further. Let $\omega_j=\ip{\eta,V_j\eta}$. Note that
    \begin{equation*}
        \omega_j=\ip{\eta,p_nV_jp_n\eta}=\ip{\eta,\sigma^\phi_{i/2}(p_nV_jp_n)\eta}
    \end{equation*}
    by choice of $\eta$. Now define $W_j=V_j-\omega_j$, $\Psi_n:B(p_n H)\rightarrow B(p_nH)$ by 
    \begin{equation*}
        \Psi_n(x)=\frac{1}{2}\sum_{j\in I}\bigg(p_nW_j^*p_nxp_nW_jp_n+\sigma^\phi_{i/2}(p_nW_jp_n)x\sigma^\phi_{-i/2}(p_nW_j^*p_n)\bigg)
    \end{equation*}
    and $k_n=(\id+\sigma_{i/2})^{-1}(\Psi_n(1))$. We want to show that the change from $V_j$ to $W_j$ does not affect the generator. Let $V\in B(p_n H)$ and $\omega\in \mathbb{C}$ be arbitrary and let $W=V-\omega$. Then we compute that
    \begin{align}
        &(\id+\sigma_{-i/2})^{-1}(W^*W)x+x(\id+\sigma_{i/2})^{-1}(W^*W)-W^*xW\nonumber\\
        =&(\id+\sigma_{-i/2})^{-1}(V^*V)x+x(\id+\sigma_{i/2})^{-1}(V^*V)-V^*xV\nonumber\\
        &-(\id+\sigma_{-i/2})^{-1}(\overline{\omega}V+\omega V^*)x+x(\id+\sigma_{i/2})^{-1}(\overline{\omega}V+\omega V^*)+\overline{\omega}xV+\omega V^*x.\label{eq:V-W-equality-1}
    \end{align}
    Likewise, denoting $\sigma_{i/2}(W)$ and $\sigma_{i/2}(V)$ by $\tilde{W}$ and $\tilde{V}$, respectively,
        \begin{align}
        &(\id+\sigma_{-i/2})^{-1}(\tilde{W}\tilde{W}^*)x+x(\id+\sigma_{i/2})^{-1}(\tilde{W}\tilde{W}^*)-\tilde{W}x\tilde{W}^*\nonumber\\
        =&(\id+\sigma_{-i/2})^{-1}(\tilde{V}\tilde{V}^*)x+x(\id+\sigma_{i/2})^{-1}(\tilde{V}\tilde{V}^*)-\tilde{V}x\tilde{V}^*\nonumber\\
        &-(\id+\sigma_{-i/2})^{-1}(\overline{\omega}\tilde{V}+\omega \tilde{V}^*)x+x(\id+\sigma_{i/2})^{-1}(\overline{\omega}\tilde{V}+\omega \tilde{V}^*)+\overline{\omega}\tilde{V}x+\omega x\tilde{V}^*.\label{eq:V-W-equality-2}
    \end{align}
    Now 
    \begin{equation*}
        V+\tilde{V}=V+\sigma_{i/2}(V)=(\id+\sigma_{i/2})(V)=(\id+\sigma_{-i/2})(V^*),
    \end{equation*}
    so 
    \begin{equation*}
        (\id+\sigma_{-i/2})^{-1}(\overline{\omega}V)+(\id+\sigma_{-i/2})^{-1}(\overline{\omega}\tilde{V})=\overline{\omega}V
    \end{equation*}
    Together with an analogous derivation for the other terms, this shows that the terms proportional to $\omega$ and $\overline{\omega}$ in Equations \ref{eq:V-W-equality-1} and \ref{eq:V-W-equality-2} add up to zero.
    
    This proves that changing $V_j$ to $V_j-\omega_j$ does not change the corresponding generator, so
    \begin{equation*}
        \cL_n(x)=k_n^*x+xk_n-\Psi_n(x),\quad x\in B(p_n H).
    \end{equation*}

    Let $\epsilon>0$ be arbitrary and let $\xi\in p_nH$ be a vector such that $\ip{\xi,\Psi_n(1)\xi}\geq \norm{\Psi_n(1)}-\epsilon$. We now consider the element $\xi\eta^*\in B(p_n H)$ defined as the map $\zeta\mapsto \xi\ip{\eta,\zeta}$. We will now lower bound the norm of $\cL_n(\xi\eta^*)$ by computing $\Re\ip{\xi,\cL_n(\xi\eta^*)\eta}$. First, observe that
    \begin{align*}
        \ip{\xi,\Psi_n(\xi\eta^*)\eta}&=\frac{1}{2}\sum_{j\in I}\ip{\xi,p_nW_j^*p_n\xi\eta^*p_nW_jp_n\eta}+\ip{\xi,\sigma^\phi_{i/2}(p_nW_jp_n)\xi\eta^*\sigma^\phi_{-i/2}(p_nW_j^*p_n)\eta}\\
        &=\frac{1}{2}\sum_{j\in I}\ip{\xi,p_nW_j^*p_n\xi}\ip{\eta,p_nW_jp_n\eta}+\ip{\xi,\sigma^\phi_{i/2}(p_nW_jp_n)\xi}\ip{\eta,\sigma^\phi_{-i/2}(p_nW_j^*p_n)\eta}\\
        &=0.
    \end{align*}
    Then
    \begin{align*}
        \Re\ip{\xi,\cL_n(\xi\eta^*)\eta}&=\Re\left(\ip{\xi,(\id+\sigma^\phi_{-i/2})^{-1}(\Phi_n(1))\xi\eta^*\eta}+\ip{\xi,\xi\eta^*(\id+\sigma^\phi_{i/2})^{-1}(\Psi_n(1))\eta}\right)\\
        &=\Re\left(\ip{\xi,(\id+\sigma^\phi_{-i/2})^{-1}(\Psi_n(1))\xi}+\ip{\eta,(\id+\sigma_{i/2})^{-1}(\Psi_n(1))\eta}\right)\\
        &=\frac{1}{2}\left(\ip{\xi,\Psi_n(1)\xi}+\ip{\eta,\Psi_n(1)\eta}\right).
    \end{align*}
    Since $\Psi_n(1)$ is positive, this shows that $\Re\ip{\xi,\cL_n(\xi\eta^*)\eta}\geq (\norm{\Psi_n(1)}-\epsilon)/2$. Since $\Psi_n$ is completely positive and $\epsilon$ was arbitrary, we conclude that $\norm{\Psi_n}\leq 2\norm{\cL_n}$.

    Next we obtain a norm bound for $k_n$. Let $\epsilon>0$ and let $\xi,\xi'\in H$ be unit vectors such that $\ip{\xi',k_n^*\xi}\geq \norm{k_n^*}-\epsilon$. By construction, we have
    \begin{equation*}
        \ip{\eta,k_n\eta}=\ip{\eta,\sigma_{-i/2}(k_n^*)\eta}=\ip{\eta,k_n^*\eta},
    \end{equation*}
    so 
    \begin{equation*}
        |\ip{\eta,k_n\eta}|=|\frac{1}{2}\ip{\eta,\Psi_n(1)\eta}|\leq \frac{1}{2}\norm{\Psi_n}.
    \end{equation*}
    On the one hand, 
    \begin{equation*}
        |\ip{\xi',(\cL_n+\Psi_n)(\xi\eta^*)\eta}|\leq \norm{\cL_n+\Psi_n}\leq 3\norm{\cL}
    \end{equation*}
    and, on the other hand,
    \begin{equation*}
        \ip{\xi',(\cL_n+\Psi_n)(\xi\eta^*)\eta}=\ip{\xi',k_n^*\xi}\ip{\eta,\eta}+\ip{\xi',\xi}\ip{\eta,k_n\eta}.
    \end{equation*}
    From this we conclude that
    \begin{align*}
        |\ip{\xi',k_n^*\xi}|&\leq |\ip{\xi',(\cL_n+\Psi_n)(\xi\eta^*)\eta}|+|\ip{\xi',\xi}\ip{\eta,k_n\eta}|\\
        &\leq 3\norm{\cL}+\frac{1}{2}|\ip{\xi',\xi}|\norm{\Psi_n}\\
        &\leq 4\norm{\cL}.
    \end{align*}

    By the Banach--Alaoglu theorem, there exist subnets $(k_{n_\alpha})$ and $(\Psi_{n_\alpha})$ and $k\in B(H)$, $\Psi\in B(B(H))$ such that $k_{n_\alpha}\to k$ in weak$^\ast$ topology and $\Psi_{n_\alpha}(p_{n_\alpha}\cdot p_{n_{\alpha}})\to \Psi$ in point-weak$^\ast$ topology. The limit $\Psi$ is a unital completely positive map that is KMS-symmetric with respect to $\phi$, hence normal. The previous norm bounds imply $\norm{k}\leq 4\norm{\cL}$ and $\norm{\Psi}\leq 2\norm{\cL}$.
    
    Moreover, passing to the limit in $p_n\cL(p_n xp_n)p_n=k_n^\ast p_n x p_n+p_n xp_n k_n-\Psi_n(p_nxp_n)$ for $x\in B(H)$, we obtain
    \begin{equation*}
        \cL(x)=k^*x+xk-\Psi(x),\quad x\in B(H).\qedhere
    \end{equation*}
\end{proof}

Based on this KMS-symmetric Christensen--Evans theorem, we can now give an explicit characterization of bounded quantum Dirichlet forms on type $\mathrm{I}$ factors.

\begin{theorem}\label{thm:char_bdd_type_I}
    If $\mathcal E$ is a bounded completely Dirichlet form on $S^2(H)$ with respect to $\rho^{1/2}$, then there exists a family $(V_j)_{j\in I}$ in $\dom(\sigma^\phi_{i/2})$ such that $\sigma_{i/2}(V_j)=V_j^\ast$ for all $j\in I$, $\sum_{j\in I}V_j^\ast V_j$ converges in strong operator topology and
    \begin{equation*}
        \mathcal E(A)=\sum_{j\in I}\int_\IR\norm{V_j\rho^{it}A\rho^{-it}-\rho^{it}A \rho^{-it}V_j^\ast}_2^2\frac{2\,dt}{\cosh 2\pi t},\quad A\in S^2(H).
    \end{equation*}
    Conversely, if $(V_j)_{j\in I}$ is a family in $\dom(\sigma^\phi_{i/2})$ such that $\sigma_{i/2}(V_j)=V_j^\ast$ for all $j\in I$ and $\sum_{j\in I}V_j^\ast V_j$ converges in strong operator topology, then
    \begin{equation*}
        \mathcal E(A)=\sum_{j\in I}\int_\IR\norm{V_j\rho^{it}A\rho^{-it}-\rho^{it}A \rho^{-it}V_j^\ast}_2^2\frac{2\,dt}{\cosh 2\pi t},\quad A\in S^2(H),
    \end{equation*}
    defines a bounded quantum Dirichlet form on $S^2(H)$ with respect to $\rho^{1/2}$.
\end{theorem}
\begin{proof}
First let $\mathcal E$ be a bounded quantum Dirichlet form and let $\cL$ be the generator of the uniformly continuous quantum Markov semigroup on $B(H)$ associated with $\mathcal E$. By \Cref{thm:CE_KMS}, there exist $k\in B(H)$ and a KMS-symmetric completely positive map $\Psi\colon B(H)\to B(H)$ such that $\cL(x)=k^\ast x+xk-\Psi(x)$ for all $x\in B(H)$. Let $\Psi(x)=\sum_{j\in I}V_j^*xV_j$ be a Kraus decomposition of $\Psi$. In particular, this means that the sum $\sum_{j\in I}V_j^\ast V_j$ converges in strong operator topology. Note that $\norm{V_j}^2=\norm{V_j^\ast V_j}\leq \norm{\Psi(1)}=\norm{\Psi}$.

The first step is to show that $V_j\in\dom(\sigma^\phi_{i/2})$ and $\norm{\sigma^\phi_{i/2}(V_j)}\leq \sqrt{\norm{\Psi}}$ for each $j\in I$. As before, let $p_n=\chi_{[1/n,n]}(\rho)$ and consider the map $\Psi_n=p_n\Psi(\cdot)p_n$ on $B(p_nH)$. It has Kraus decomposition $\Psi_n(x)=\sum_{j\in I}(p_n V_j p_n)^\ast x(p_n V_j p_n)$.

By KMS symmetry,
\begin{align*}
    \Tr(\Psi_n(x)\rho^{1/2}y\rho^{1/2})&=\sum_j\Tr(x\rho^{1/2}p_nV_j^\ast y V_j p_n\rho^{1/2})\\
    &=\sum_j\Tr(\sigma^\phi_{i/2}(p_n V_j p_n) x \sigma^\phi_{i/2}(p_n V_j p_n)^\ast\rho^{1/2}y\rho^{1/2})
\end{align*}
for $x,y\in B(p_nH)$. In particular,
\begin{equation*}
    \norm{\sigma^\phi_{i/2}(p_n V_j p_n)}^2=\norm{\sigma^\phi_{i/2}(p_n x p_n)\sigma^\phi_{i/2}(p_n x p_n)^\ast}\leq \norm{\Psi_n(1)}\leq \norm{\Psi}.
\end{equation*}
Since $p_n V_j p_n\to V_j$ in weak$^\ast$ topology and $\sigma^\phi_{i/2}$ is weak$^\ast$ closed by \cite[Theorem 2.4]{CZ76}, it follows that $V_j\in \dom(\sigma^\phi_{i/2})$ and $\norm{\sigma^\phi_{i/2}(V_j)}\leq \sqrt{\norm{\Psi}}$. Since $\dom(\sigma^\phi_{i/4})\subset \dom(\sigma^\phi_{i/2})$ by definition, we conclude that $V_j\in \dom(\sigma^\phi_{i/4})$ and the Phragmen--Lindelöf theorem implies $\norm{\sigma^\phi_{i/4}(V_j)}\leq \max\{\norm{V_j},\norm{\sigma^\phi_{i/2}(V_j)}\}\leq \sqrt{\norm{\Psi}}$.

Now define $V_j^{+}=e^{i\pi/4}V_j+e^{-i\pi/4}\sigma^\phi_{i/2}(V_j)^*$ and $V_j^-=e^{-i\pi/4}V_j+e^{i\pi/4}\sigma^\phi_{i/2}(V_j)^*$. A direct computation shows
\begin{align*}
    V_j^{+*}xV_j^{+}+V_j^{-*}xV_j^-
    =&2V_j^*xV_j+2\sigma_{i/2}(V_j)x\sigma_{i/2}(V_j)^*
\end{align*}
for all $x\in B(H)$.

Note that $x\mapsto \sigma^\phi_{i/2}(V_j)x\sigma^\phi_{i/2}(V_j)^*$ is the KMS-adjoint of the map $x\mapsto V_j^*xV_j$ with respect to the state $\phi$. As a result,
\begin{equation*}
    \sum_{j\in I}V_j^{+*}xV_j^++V_j^{-*}xV_j^-=4\Psi(x)
\end{equation*}
Therefore, we can assume without loss of generality that the operator $V_j$ satisfies $V_j=\sigma_{i/2}^\phi(V_j)^*$ for every $j\in I$.

Let $x\in \rho^{1/4} B(H)\rho^{1/4}$. First, we calculate that
\begin{align*}
    \norm{\rho^{1/4}[\sigma^{\phi}_{i/4}(V_j),x]\rho^{1/4}}^2=&\Tr\left(\rho^{1/2}[\sigma^{\phi}_{i/4}(V_j),x]^*\rho^{1/2}[\sigma^{\phi}_{i/4}(V_j),x]\right)\\
    =&\Tr\big(\rho^{1/2}x^*\sigma^{\phi}_{i/4}(V_j)^*\rho^{1/2}\sigma^{\phi}_{i/4}(V_j)x+\rho^{1/2}\sigma^{\phi}_{i/4}(V_j)^*x^*\rho^{1/2}x\sigma^{\phi}_{i/4}(V_j)\\
    &-\rho^{1/2}\sigma^{\phi}_{i/4}(V_j)^*x^*\rho^{1/2}\sigma^{\phi}_{i/4}(V_j)x-\rho^{1/2}x^*\sigma^{\phi}_{i/4}(V_j)^*\rho^{1/2}x\sigma^{\phi}_{i/4}(V_j)\big)\\
    =&\Tr\big(\rho^{1/2}x^*\rho^{1/4}V_j^*V_j\rho^{1/4}x+\sigma^{\phi}_{i/2}(V_j)\sigma^{\phi}_{-i/2}(V_j^*)\rho^{1/4}x^*\rho^{1/2}x\rho^{1/4}\big)\\
    &-\rho^{1/2}V_j^*\rho^{-1/4}x^*\rho^{1/4}V_j\rho^{1/2}\rho^{-1/4}x\rho^{1/4}\\
    &-\rho^{1/2}\rho^{1/4}x^*\rho^{-1/4}\rho^{1/2}V_j^*\rho^{1/4}x\rho^{-1/4}V_j\big).
\end{align*}
Taking the sum over all $j\in I$ and using that $k^*+k=\sum_{j\in I}V_j^*V_j=\sum_{j\in I}\sigma^{\phi}_{i/2}(V_j)\sigma^{\phi}_{-i/2}(V_j^*)$, we find that
\begin{align*}
    \sum_{j\in I}\norm{\rho^{1/4}[\sigma^{\phi}_{i/4}(V_j),x]\rho^{1/4}}^2=&\Tr\big(\rho^{1/2}x^*\rho^{1/4}(k^*+k)\rho^{1/4}x+(k^*+k)\rho^{1/4}x^*\rho^{1/2}x\rho^{1/4}\big)\\
    &-\rho^{1/2}\Phi(\rho^{-1/4}x^*\rho^{1/4})\rho^{1/2}\rho^{-1/4}x\rho^{1/4}\\
    &-\rho^{1/2}\rho^{1/4}x^*\rho^{-1/4}\rho^{1/2}\Phi(\rho^{1/4}x\rho^{-1/4})\big)\\
    =&\Tr\left(\rho^{1/2}(k^*\rho^{-1/4}x^*\rho^{1/4}+\rho^{-1/4}x^*\rho^{1/4}k)\rho^{1/2}\rho^{-1/4}x\rho^{1/4}\right)\\
    &-\Tr\left(\rho^{1/2}\Phi(\rho^{-1/4}x^*\rho^{1/4})\rho^{1/2}\rho^{-1/4}x\rho^{1/4}\right)\\
    &+\Tr\left(\rho^{1/2}\rho^{1/4}x\rho^{-1/4}\rho^{1/2}(k^*\rho^{1/4}x\rho^{-1/4}+\rho^{1/4}x\rho^{-1/4}k)\right)\\
    &-\Tr\left(\rho^{1/2}\rho^{1/4}x^*\rho^{-1/4}\rho^{1/2}\Phi(\rho^{1/4}x\rho^{-1/4})\right)\\
    =&\Tr\left(\rho^{1/2}\cL(\rho^{-1/4}x^*\rho^{1/4})\rho^{1/2}\rho^{-1/4}x\rho^{1/4}\right)\\
    &+\Tr\left(\rho^{1/2}\rho^{1/4}x^*\rho^{-1/4}\rho^{1/2}\cL(\rho^{1/4}x\rho^{-1/4})\right).
\end{align*}

Consequently, 
\begin{align*}
    \sum_{j\in I}\norm{\rho^{1/4}[\sigma^{\phi}_{i/4}(V_j),\rho^{it}x\rho^{-it}]\rho^{1/4}}^2=&\Tr\left(\rho^{1/2}\cL(\rho^{-1/4+it}x^*\rho^{1/4-it})\rho^{1/2}\rho^{-1/4+it}x\rho^{1/4-it}\right)\\
    &+\Tr\left(\rho^{1/2}\rho^{1/4+it}x^*\rho^{-1/4-it}\rho^{1/2}\cL(\rho^{1/4+it}x\rho^{-1/4-it})\right).
\end{align*}

By the KMS-symmetry of $\cL$ and Lemma \ref{lem:appl_residue_thm}, this implies that
\begin{align*}
	&\int_{\IR}\sum_{j\in I}\norm{V_j \rho^{it+1/4}x\rho^{-it+1/4}-\rho^{it+1/4}x\rho^{-it+1/4}V_j^\ast}_2^2\frac{2\,dt}{\cosh 2\pi t}\\
    &\qquad=\int_{\IR}\sum_{j\in I}\norm{\rho^{1/4}[\sigma^{\phi}_{i/4}(V_j),\rho^{it}x\rho^{-it}]\rho^{1/4}}_2^2\frac{2\,dt}{\cosh 2\pi t}\\
    &\qquad=\Tr(\rho^{1/2}x^*\rho^{1/2}\cL(x))\\
    &\qquad=\mathcal E(\rho^{1/4}x\rho^{1/4})
\end{align*}
for all $x\in \rho^{1/4}B(H)\rho^{1/4}$. Since $\rho^{1/2}B(H)\rho^{1/2}$ is dense in $S^2(H)$ and $\mathcal E$ is bounded, this settles the first implication.

For the converse, consider the Hilbert space $\cH=L^2(\IR,\mu;\bigoplus_{j\in I}S^2(H))$ with bimodule actions of $B(H)$ given by $(x\eta y)(t)=(\sigma^\phi_t(x)\eta_j(t)\sigma^\phi_t(y))_{j\in I}$ and $(\cJ \eta)(t)=(\eta_j(t)^\ast)_{j\in I}$, $(\cV_s\eta)(t)=\eta(s+t)$. One can check directly that $(\cH,\cJ)$ is a self-dual $B(H)$-$B(H)$ correspondence and $\cV_s$ satisfies $\cV_s\cJ=\cJ \cV_s$ and $\cV_s(x\eta y)=\sigma^\phi_s(x)(\cV_s\eta)\sigma^\phi_s(y)$.

Let 
\begin{equation*}
    \delta\colon S^2(H)\to \cH,\,(\delta A)(t)=(V_j \rho^{it}A\rho^{-it}-\rho^{it}A\rho^{-it}V_j^\ast)_{j\in I}.
\end{equation*}
Since $\sum_j V_j^\ast V_j$ converges in strong operator topology, we have
\begin{align*}
    \int_{\IR}\sum_{j}\norm{V_j \rho^{it}A\rho^{-it}-\rho^{it}A\rho^{-it}V_j^\ast}_2^2\frac{2\,dt}{\cosh 2\pi t}\leq 4\left\lVert\sum_{j\in I} V_j^\ast V_j\right\rVert\norm{A}_2^2,
\end{align*}
so that $\delta$ does indeed take values in $\cH$ and is bounded. It is easy to see that $\delta\circ \Delta_\phi^{is}=\cV_s\circ\delta$ and $\delta\circ J=\cJ\circ \delta$. Finally, if $x,y\in B(H)$ and $t\in \IR$, then
\begin{align*}
    \delta(x\rho^{1/2}y)(t)&=(V_j\sigma^\phi_t(x)\rho^{1/2}\sigma^\phi_t(y)-\sigma^\phi_t(x)\rho^{1/2}\sigma^\phi_t(y) V_j^\ast)_{j\in I}\\
    &=((V_j \Delta_\phi^{it}x\rho^{1/2}-(\Delta_\phi^{it}x\rho^{1/2})V_j^\ast)y+\sigma^\phi_t(x)\rho^{1/2}(V_j^\ast \sigma^\phi_t(y)-\sigma^\phi_t(y)V_j^\ast))_{j\in I}\\
    &=(\delta(x\rho^{1/2})y)(t)+(\sigma^\phi_t(x) (V_j \Delta_\phi^{it}\rho^{1/2}y-(\Delta_\phi^{it}\rho^{1/2}y)V_j))_{j\in I}\\
    &=(\delta(x\rho^{1/2})y)(t)+(x\delta(\rho^{1/2}y))(t).
\end{align*}
Now it follows from \Cref{thm:char_bdd} that
\begin{equation*}
    \mathcal E(A)=\sum_{j\in I}\int_\IR\norm{V_j\rho^{it}A\rho^{-it}-\rho^{it}A \rho^{-it}V_j^\ast}_2^2\frac{2\,dt}{\cosh 2\pi t},\quad A\in S^2(H),
\end{equation*}
defines a quantum Dirichlet form.
\end{proof}

\section{\texorpdfstring{The Tomita algebra $\mathfrak A_{\mathcal E}$}{The Tomita algebra A_E}}\label{sec:Tomita_alg}

In this section we explore some consequences of the characterization of bounded quantum Dirichlet forms from \Cref{sec:char_bdd} for the domains of unbounded quantum Dirichlet forms. In particular, we introduce the space $\mathfrak A_{\mathcal E}$ for a quantum Dirichlet form $\mathcal E$, show that it is a Tomita algebra and a form core for $\mathcal E$. Moreover, we show that the modular group restricts to a strongly continuous group on the domain of $\mathcal E$ such that the square root of its analytic generator is accretive with respect to the form inner product.

\begin{lemma}\label{lem:bound_mod_grp_form}
    If $\mathcal E$ is a conservative completely Dirichlet form and $a\in L^2(M)$, then $\mathcal E(\Delta_\phi^{it}a)\leq e^{2\pi\abs{t}}\mathcal E(a)$ for all $t\in\IR$. In particular, $\dom(\mathcal E)$ is invariant under $\Delta_\phi^{it}$ for all $t\in \IR$.
\end{lemma}
\begin{proof}
    First assume that $\mathcal E$ is bounded and $a\in M\Omega\cap M^\prime\Omega$. By \Cref{thm:int_rep_form}, there exists an $M$-$M$ correspondence $\mathcal H$ and a linear map $\delta\colon M\Omega\cap M^\prime\Omega\to \mathcal H$ such that 
    \begin{equation*}
        \mathcal E(\Delta_\phi^{it}a)=\int_{\IR}\norm{\delta(\Delta_\phi^{i(s+t)}a)}_{\mathcal H}^2\frac{2\,ds}{\cosh2\pi s}.
    \end{equation*}
    Since $\frac{\cosh 2\pi s}{\cosh 2\pi(s-t)}\in [e^{-2\pi\abs{t}},e^{2\pi\abs t}]$ for all $s,t\in \IR$, we have
    \begin{align*}
        \mathcal E(\Delta_\phi^{it}a)&=\int_\IR\norm{\delta(\Delta_\phi^{is}a)}_{\mathcal H}^2\frac{2\,ds}{\cosh 2\pi(s-t)}\\
        &\leq e^{2\pi\abs t}\int_{\IR}\norm{\delta(\Delta_\phi^{is}a)}_{\mathcal H}^2\frac{2\,ds}{\cosh 2\pi s}\\
        &=e^{2\pi\abs t}\mathcal E(a).
    \end{align*}
    The result for general $a\in L^2(M)$ follows since $M\Omega\cap M^\prime \Omega$ is dense in $L^2(M)$ and $\mathcal E$ is bounded.

    Now if $\mathcal E$ is not necessarily bounded, we can use that the Moreau regularization $\mathcal E^{(\epsilon)}$ is a bounded completely Dirichlet form for all $\epsilon>0$ to obtain
    \begin{equation*}
        \mathcal E(\Delta_\phi^{it} a)=\sup_{\epsilon>0}\mathcal E_\epsilon(\Delta_\phi^{it}a)\leq e^{2\pi\abs{t}}\sup_{\epsilon>0}\mathcal E_\epsilon(a)=e^{2\pi \abs{t}}\mathcal E(a)
    \end{equation*}
    for all $a\in L^2(M)$ and $t\in\IR$.
\end{proof}

\begin{lemma}\label{lem:res_mod_group_form}
    If $\mathcal E$ is a conservative completely Dirichlet form, then $(\Delta_\phi^{it})_{t\in\IR}$ leaves $\dom(\mathcal E)$ globally invariant and its restriction to $(\dom(\mathcal E),\norm\cdot_{\mathcal E})$ is a strongly continuous group.
\end{lemma}
\begin{proof}
    We have already seen in \Cref{lem:bound_mod_grp_form} that $\dom(\mathcal E)$ is globally invariant under $(\Delta_\phi^{it})_{t\in\IR}$ and $\norm{\Delta_\phi^{it}a}_{\mathcal E}\leq e^{\pi\abs{t}}\norm{a}_{\mathcal E}$ for all $a\in\dom(\mathcal E)$ and $t\in\IR$. Let $L$ be the positive self-adjoint operator associated with $\mathcal E$. To show strong continuity, it is therefore \cite[Theorem I.5.8]{EN00} enough to prove that $t\mapsto \mathcal E(b,\Delta_\phi^{it}a)$ is continuous at $t=0$ for all $b\in \dom(L)$. This follows immediately from $\mathcal E(b,\Delta_\phi^{it}a)=\langle Lb,\Delta_\phi^{it}a\rangle_2$ and the strong continuity of $(\Delta_\phi^{it})$ with respect to the $L^2$-norm.
\end{proof}

\begin{definition}
    If $\mathcal E$ is a conservative completely Dirichlet form, let $\mathfrak A_{\mathcal E}$ be the subset of $L^2(M)$ consisting of all elements $x\Omega_\phi$ for which $x\in M$ is analytic for $\sigma^\phi$, $\sigma^\phi_z(x)\Omega_\phi\in \dom(\mathcal E)$ for all $z\in\IC$ and the map $z\mapsto \sigma^\phi_z(x)\Omega_\phi$ is entire analytic with respect to the form norm $\norm\cdot_{\mathcal E}$. We call $\mathfrak A_{\mathcal E}$ the \emph{Tomita algebra associated with $\mathcal E$}. Moreover, let $A_{\mathcal E}$ denote the norm closure of $\{x\in M\mid x\Omega_\phi\in \mathfrak A_{\mathcal E}\}$.
\end{definition}

\begin{remark}
    The name ``Tomita algebra associated with $\mathcal E$'' is justified by \Cref{cor:Tomita_subalgebra} below. Note that in the GNS-symmetric case, it suffices to require that $z\mapsto \sigma^\phi_z(x)\Omega_\phi$ is analytic for the $L^2$-norm by \cite[Lemma 6.6]{Wir22b}. Thus the definition of $\mathfrak A_{\mathcal E}$ from \cite[Definition 6.4]{Wir22b} coincides with the one given here in the GNS-symmetric case.
\end{remark}

\begin{proposition}\label{prop:Tomita_form_dense}
    If $\mathcal E$ is a quantum Dirichlet form, then $\mathfrak A_{\mathcal E}$ is a form core for $\mathcal E$.
\end{proposition}
\begin{proof}
    If $a\in L^2(M)$ and $Ja=a$, let $a_n$ be the projection onto the closed convex subset $\{b\in L^2(M)\mid Jb=b,\,n\Omega_\phi\pm b\in L^2_+(M)\}\subset \Delta_\phi^{1/4}M\Omega_\phi$. We have $a_n\to a$ in $L^2(M)$ and $\mathcal E(a_n)\leq \mathcal E(a)$ by the Dirichlet form property \cite[Proposition 4.5]{Cip97}. It follows that every subnet of $(a_n)$ has a subnet that converges weakly to $a$ in $(\dom(\mathcal E),\norm\cdot_{\mathcal E})$, hence $(a_n)$ itself converges weakly to $a$ in $(\dom(\mathcal E),\norm\cdot_{\mathcal E})$. Together with $\limsup_{n\to\infty}\norm{a_n}_{\mathcal E}\leq \norm{a}_{\mathcal E}$, this implies $\norm{a-a_n}_{\mathcal E}\to 0$.
    
    Since $\dom(\mathcal E)$ is invariant under $J$, it suffices therefore to show that for every $a\in\dom(\mathcal E)\cap \Delta_\phi^{1/4}M\Omega_\phi$ there exists a sequence $(a_n)$ in $\mathfrak A_{\mathcal E}$ such that $a_n\to a$ in form norm.

    Let $x\in M$ such that $a=\Delta_\phi^{1/4}x\Omega_\phi\in\dom(\mathcal E)$ and let $ x_n=\sqrt{\frac n\pi}\int_{\IR}e^{-nt^2}\sigma^\phi_t(x)\,dt$. It follows from the proof of \cite[Lemma 1.3]{Haa75} that $x_n$ is analytic for $\sigma^\phi$ and $x_n\Omega_\phi\to x\Omega_\phi$, $x_n^\ast \Omega_\phi\to x^\ast\Omega_\phi$. By Young's inequality, we obtain
    \begin{align*}
        \norm{\Delta_\phi^{1/4}(x-x_n)\Omega_\phi}^2
        &=\langle (x-x_n)\Omega_\phi,\Delta_\phi^{1/2}(x-x_n)\Omega_\phi\rangle\\
        &\leq \frac 1 2\norm{(x-x_n)\Omega_\phi}^2+\frac 1 2\norm{\Delta_\phi^{1/2}(x-x_n)\Omega_\phi}^2\\
        &=\frac 1 2\norm{(x-x_n)\Omega_\phi}^2+\frac 1 2\norm{(x^\ast-x_n^\ast)\Omega_\phi}^2\\
        &\to 0.
    \end{align*}
    
    Let $V_t=\Delta_\phi^{it}|_{\dom(\mathcal E)}$. By \Cref{lem:res_mod_group_form}, the family $(V_t)_{t\in\IR}$ is a strongly continuous group on $(\dom(\mathcal E),\norm\cdot_{\mathcal E})$. We have
    \begin{equation*}
        \Delta_\phi^{1/4}x_n\Omega_\phi=\sqrt\frac n\pi\int_\IR e^{-nt^2}V_t(\Delta_\phi^{1/4}x\Omega_\phi)\,dt\in \bigcap_{z\in\IC}\dom(V_z)
    \end{equation*}
    by \Cref{lem:Haagerup_approx_groups}. Thus $\Delta_\phi^{1/4}x_n\Omega_\phi=\sigma_{-i/4}^\phi(x_n)\Omega_\phi\in \mathfrak A_{\mathcal E}$. Finally,
    \begin{align*}
        \norm{\Delta_\phi^{1/4}x\Omega_\phi-\Delta_\phi^{1/4}x_n\Omega_\phi}_{\mathcal E}
        &\leq \sqrt\frac n\pi\int_{\IR}e^{-nt^2}\norm{V_t\Delta_\phi^{1/4}x\Omega_\phi-\Delta_\phi^{1/4}x\Omega_\phi}_{\mathcal E}\,dt\\
        &=\frac 1{\sqrt \pi}\int_\IR e^{-s^2}\norm{V_{s/\sqrt n}\Delta_\phi^{1/4}x\Omega_\phi-\Delta_\phi^{1/4}x\Omega_\phi}_{\mathcal E}\,ds\\
        &\to 0
    \end{align*}
    by the dominated convergence theorem.
\end{proof}

\begin{theorem}\label{thm:analytic_gen_mod_grp_accretive}
    If $\mathcal E$ is a quantum Dirichlet form and $(V_s)_{s\in\IR}$ denotes the restriction of $(\Delta_\phi^{is})_{s\in\IR}$ to $(\dom(\mathcal E),\norm{\cdot}_{\mathcal E})$, then $V_{-i/2}$ is accretive.
\end{theorem}
\begin{proof}
    By \Cref{prop:Tomita_form_dense}, the set $\mathfrak A_{\mathcal E}$ is dense in $(\dom(\mathcal E),\norm{\cdot}_{\mathcal E})$ and by definition it is contained in $\dom(V_{-i/2})$. It follows from \Cref{lem:core_analytic_generator} that it is a core for $V_{-i/2}$.
    
    It is clear from the definition that $V_{-i/2}\subset \Delta_\phi^{1/2}$. Since $\langle a,\Delta_\phi^{1/2}a\rangle_2\geq 0$ for all $a\in \dom(\Delta_\phi^{1/2})$, it suffices to show that $\Re\mathcal E(a,V_{-i/2}a)\geq 0$ for all $a\in\mathfrak A_{\mathcal E}$. By polarization, $\mathcal E^{(\epsilon)}(b,c)\to \mathcal E(b,c)$ for all $b,c\in\dom(\mathcal E)$. Therefore we can assume that $\mathcal E$ is bounded.

    Let $(\cH,\cJ,(\cV_s)_{s\in\IR},\delta)$ as in \Cref{thm:char_bdd}. By \cite[Theorem 4.2]{Uit00}, we have
    \begin{equation*}
        V_{-i/2}a=\lim_{n\to\infty}\sqrt\frac n\pi\int_{\IR}e^{-n(s+i/2)^2}\Delta_\phi^{is}a\,ds.
    \end{equation*}
    Since $\delta$ is bounded, this implies
    \begin{equation*}
        \delta(V_{-i/2}a)=\lim_{n\to\infty}\sqrt\frac n\pi\int_{\IR}e^{-n(s+i/2)^2}\cV_s\delta(a)\,ds,
    \end{equation*}
    which equals $\cV_{-i/2}\delta(a)$ by another application of \cite[Theorem 4.2]{Uit00}.
    
    Thus $\Re\mathcal E(a,V_{-i/2}a)=\Re\langle \delta(a),\cV_{-i/2}\delta(a)\rangle$, which is positive since $\cV_{-i/2}$ is accretive.
\end{proof}

\begin{remark}
    The preceding theorem can also be checked directly without using \Cref{thm:char_bdd}. By the approximation argument given in the proof above, it suffices to consider bounded quantum Dirichlet forms. If $\mathcal E$ is bounded, then $\norm\cdot_{\mathcal E}$ and $\norm\cdot_2$ are equivalent, hence $V_{-i/2}=\Delta_\phi^{1/2}$ and $\Delta_\phi^{1/4}\dom(\sigma^\phi_{-i/2})\Omega_\phi$ is a core for $V_{-i/2}$. Let $(\Phi_t)_{t\geq 0}$ be the quantum Markov semigroup associated with $\mathcal E$ and let $\cL$ denote its generator. If $x\in \dom(\sigma^\phi_{-i/2})$, then
    \begin{align*}
        \Re\mathcal E(\Delta_\phi^{1/4}x\Omega_\phi,\Delta_\phi^{3/4}x\Omega_\phi)
        &=\Re\langle \Delta_\phi^{1/4}\cL(x)\Omega_\phi,\Delta_\phi^{3/4}x\Omega_\phi\rangle\\
        &=\Re\langle \Delta_\phi^{1/2}\cL(x)\Omega_\phi,\Delta_\phi^{1/2}x\Omega_\phi\rangle\\
        &=\Re\langle x^\ast\Omega_\phi,\cL(x)^\ast\Omega_\phi\rangle\\
        &=\frac 1 2\phi(x\cL(x)^\ast+\cL(x)x^\ast)\\
        &=\frac1  2\lim_{t\to 0}\frac 1 t\phi(x x^\ast-\Phi_t(x)\Phi_t(x)^\ast)
    \end{align*}
    The last expression is positive by the Kadison--Schwarz inequality.
\end{remark}

\begin{lemma}
    If $\mathcal E$ is a quantum Dirichlet form and $a\in \dom(\mathcal E)\cap M\Omega_\phi$, $b\in \dom(\mathcal E)\cap M^\prime \Omega_\phi$, then $ab\in \dom(\mathcal E)$ and
    \begin{equation*}
        \mathcal E(ab)^{1/2}\leq \norm{\pi_l(a)}\mathcal E(b)^{1/2}+\mathcal E(a)^{1/2}\norm{\pi_r(b)}.
    \end{equation*}
\end{lemma}
\begin{proof}
    First assume that $\mathcal E$ is bounded. By \Cref{thm:char_bdd} there exists an $M$-$M$ correspondence $\cH$ and a bounded linear map $\delta\colon L^2(M)\to\cH$ such that $\delta(ab)=\pi_l(a)\cdot\delta(b)+\delta(a)\cdot J \pi_r(b)^\ast J$ and $\mathcal E(c)=\norm{\delta(c)}^2$ for all $c\in L^2(M)$. Thus
    \begin{align*}
        \mathcal E(ab)^{1/2}
        &=\norm{\delta(ab)}\\
        &=\norm{\pi_l(a)\delta(b)+\delta(a)\cdot J \pi_r(b)^\ast J}\\
        &\leq \norm{\pi_l(a)}\mathcal E(b)^{1/2}+\mathcal E(a)^{1/2}\norm{\pi_r(b)}.
    \end{align*}
    For not necessarily bounded forms, the result follows now easily using the approximation by the Moreau regularization as before.
\end{proof}

\begin{corollary}\label{cor:Tomita_subalgebra}
    If $\mathcal E$ is a conservative completely Dirichlet form, then $\mathfrak A_{\mathcal E}$ is a Tomita subalgebra of the maximal Tomita algebra associated with $\phi$.
\end{corollary}

\begin{lemma}
    If $\mathcal E$ is a quantum Dirichlet form, then $A_{\mathcal E}$ is a strong$^\ast$ dense unital $C^\ast$-subalgebra of $M$.
\end{lemma}
\begin{proof}
    It follows from Corollary \ref{cor:Tomita_subalgebra} that $\pi_l(\mathfrak A_{\mathcal E})$ is a unital $\ast$-subalgebra of $M$. Let $N=\pi_l(\mathfrak A_{\mathcal E})^{\prime\prime}$. As a consequence of Lemma \ref{lem:res_mod_group_form}, the set $\pi_l(\mathfrak A_{\mathcal E})$ and thus also $N$ is invariant under $\sigma^\phi$. By Takesaki's theorem, there exists a unique $\phi$-preserving conditional expectation $E\colon M\to N$, which has the property that $E(x)\Omega_\phi$ is the orthogonal projection of $x\Omega_\phi$ onto $\overline{N\Omega_\phi}$ for $x\in M$. However, since $\mathfrak A_{\mathcal E}$ is a form core for $\mathcal E$, it is in particular dense in $L^2(M)$. Therefore $E=\mathrm{id}$ and $N=M$. The strong$^\ast$ density of $A_{\mathcal E}=\overline{\pi_l(\mathfrak A_{\mathcal E})}^{\norm\cdot}$ in $M$ is now a consequence of the Kaplansky density theorem.
\end{proof}

\section{Unbounded quantum Dirichlet forms}\label{sec:unbdd_Dir_forms}

In this section we prove a representation theorem for general, not necessarily bounded quantum Dirichlet forms. This result takes on a similar form to \Cref{thm:char_bdd} with the exception that the derivation $\delta$ of course need not be bounded and the bimodule actions are only defined on the $C^\ast$-algebra $A_{\mathcal E}$ and are not necessarily bounded. The proof proceeds through an ultraproduct construction, the hard part being the accretivity of the square root of the analytic generator of the ultraproduct of the strongly continuous groups.

As before, throughout this section $M$ is a von Neumann algebra and $\phi$ a faithful normal state on $M$.

Let $\mathcal E$ be a quantum Dirichlet form on $L^2(M)$ with respect to $\Omega_\phi$. By \Cref{lem:Moreau_reg_Dir}, for every $n\in\IN$ the Moreau regularization $\mathcal E_{1/n}$ is a bounded quantum Dirichlet form. By \Cref{thm:char_bdd} we obtain for every $n\in\IN$ a quadruple $(\cH^{(n)},\cJ^ {(n)},(\cV^{(n)}_s)_{s\in \IR},\delta^{(n)})$ consisting of a self-dual $M$-$M$ correspondence $(\cH^{(n)},\cJ^{(n)})$, a strongly continuous group $(\cV_s^{(n)})_{s\in \IR}$ on $\cH^{(n)}$ and a bounded linear operator $\delta^{(n)}\colon L^2(M)\to \cH^{(n)}$ satisfying the properties (i)--(v) stated there.

If $\omega\in\beta\mathbb N\setminus\mathbb N$, consider the ultraproduct $\prod_{n\to\omega}\cH^{(n)}$, which is an $A_{\mathcal E}$-$A_{\mathcal E}$ correspondence with the bimodule actions defined by $x(\eta_n)_n^\omega y=(x\eta_n y)_n^\omega$ for $x,y\in A_{\mathcal E}$ and $(\eta_n)_n^\omega\in \prod_{n\to\omega}\cH_n$. Note that while the bimodule actions are well-defined for all $x,y\in M$, they may fail to be normal so that $\prod_{n\to\omega}\cH^{(n)}$ is not necessarily a (von Neumann algebraic) $M$-$M$ correspondence.

We define $\delta\colon \mathfrak A_{\mathcal E}\to \prod_{n\to\omega}\cH^{(n)}$, $\delta(a)=(\delta^{(n)}(a))_n^\omega$ and $\cH=\overline{\operatorname{span}}\{x\delta(a)y\mid a\in \mathfrak A_{\mathcal E},x,y\in A_{\mathcal E}\}$. Note that $\mathcal H$ is still an $A_{\mathcal E}$-$A_{\mathcal E}$ correspondence.

\begin{lemma}\label{lem:deriv_unbdd}
    The operator $\delta\colon \mathfrak A_{\mathcal E}\to \cH$ is closable, $\dom(\bar\delta)=\dom(\mathcal E)$ and $\norm{\bar\delta(a)}^2=\mathcal E(a)$ for all $a\in \dom(\mathcal E)$.
\end{lemma}
\begin{proof}
If $a\in\mathfrak A_{\mathcal E}$, then
\begin{equation*}
    \norm{\delta(a)}^2=\lim_{n\to\omega}\norm{\delta^{(n)}(a)}^2=\lim_{n\to\omega}\mathcal E_{1/n}(a)=\mathcal E(a).
\end{equation*}
As $\mathfrak A_{\mathcal E}$ is a form core for $\mathcal E$ by \Cref{prop:Tomita_form_dense}, the remaining claims follow from standard operator-theoretic arguments.
\end{proof}

Define bounded linear maps $\cJ,\cV_s\colon \cH\to \prod_{n\to\omega}\cH^{(n)}$ by $\cJ((\eta_n)_n^\omega)=(\cJ^{(n)} \eta_n)_n^\omega$ and $\cV_s((\eta_n)_n^\omega)=(\cV_s^{(n)} \eta_n)_n^\omega$. Note that $\cV_s$ is well-defined and $\norm{\cV_s}\leq e^{\pi\abs{s}}$ since $\norm{\cV_s^{(n)}}\leq e^{\pi\abs{s}}$ for all $n\in\IN$ by \Cref{rem:norm_V_s}.

The maps $\cJ$ and $\cV_s$, $s\in\IR$, take values in $\cH$. For $\cJ$, this follows from
\begin{align*}
    \cJ(x\delta(a)y)=(\cJ^{(n)}(x\delta^{(n)}(a)y))_n^\omega=(y^\ast \delta^{(n)}(J a)x^\ast)_n^\omega=y^\ast \delta(J a)x^\ast
\end{align*}
for $x,y\in A_{\mathcal E}$ and $a\in \mathfrak A_{\mathcal E}$. For $\cV_s$, the proof is similar.

\begin{lemma}
    The one-parameter group $(\cV_s)_{s\in\IR}$ on $\cH$ is strongly continuous.
\end{lemma}
\begin{proof}
    Since $\norm{\cV_s}\leq e^{\pi\abs{s}}$, by \cite[Theorem I.5.8]{EN00} it suffices to show that $\cV_s$ is weakly continuous on the total subset $\{x\delta(a)y\mid a\in \mathfrak A_{\mathcal E},\,x,y\in A_{\mathcal E}\}$ of $\cH$. Since $\cV_s(x\delta(a)y)=\sigma^\phi_s(x)\delta(\Delta_\phi^{is}a)\sigma^\phi_s(y)$ and $\sigma^\phi_s(x)\to x$, $\sigma^\phi_s(y)\to y$ in strong operator topology as $s\to 0$, it is enough to show that $\delta(\Delta_\phi^{is}a)\to \delta(a)$ weakly as $s\to 0$.
    
    We have $\norm{\delta(\Delta_\phi^{is}a)}=\mathcal E(\Delta_\phi^{is}a)^{1/2}\leq e^{\pi\abs{s}}\norm{\delta(a)}$ by \Cref{lem:bound_mod_grp_form}. Let $L$ be the positive self-adjoint operator associated with $\mathcal E$. If $b\in \dom(L)$, then it follows from \Cref{lem:deriv_unbdd} that
\begin{equation*}
    \langle \delta(\Delta_\phi^{is}a),\overline\delta(b)\rangle=\langle \Delta_\phi^{is}a,L(b)\rangle\to \langle a,L(b)\rangle=\langle \delta(a),\overline\delta(b)\rangle\quad \text{as }s\to 0.
\end{equation*}
Since $\dom(L)$ is dense in $\dom(\mathcal E)$ with respect to $\norm{\cdot}_{\mathcal E}$ by the spectral theorem, $\overline\delta(\dom(L))$ is dense in $\overline{\delta(\mathfrak A_{\mathcal E})}$. We conclude that $\delta(\Delta_\phi^{is}a)\to \delta(a)$ weakly as $s\to 0$.
\end{proof}

The hard remaining part is to check that $\cV_{-i/2}$ is still accretive. To do so, we need two preliminary lemma which allow us to check accretivity on a convenient subset of $\dom(\cV_{-i/2})$.

\begin{lemma}\label{lem:anatytic_ext_V}
    If $a,b\in \mathfrak A_{\mathcal E}$, then $\pi_l(a)\cdot\delta(b)\in \bigcap_{z\in\mathbb C}\dom(\mathcal V_z)$ and $\mathcal V_z(\pi_l(a)\cdot\delta(b))=\pi_l(\Delta_\phi^{iz}a)\cdot\delta(\Delta_\phi^{iz}b)$ for all $z\in\IC$.
\end{lemma}
\begin{proof}
    From the properties of $(\mathcal V_t)_{t\in\IR}$, it follows that $\mathcal V_t(\pi_l(a)\cdot\delta(b))=\pi_l(\Delta_\phi^{it}a)\cdot\delta(\Delta_\phi^{it}b)$ for $t\in\IR$. Hence it suffices to prove that $z\mapsto \pi_l(\Delta_\phi^{iz}a)\cdot\delta(\Delta_\phi^{iz}b)$ is entire analytic. By definition of $\mathfrak A_{\mathcal E}$, we have $\pi_l(a)\in \dom(\sigma^\phi_z)$ and $\sigma^\phi_z(\pi_l(a))=\pi_l(\Delta_\phi^{iz}a)$ for $z\in \IC$. Moreover, the map $z\mapsto \sigma^\phi_z(\pi_l(a))$ is entire analytic. By the product rule, it is therefore enough to show that $z\mapsto \delta(\Delta_\phi^{iz}b)$ is entire analytic.

    By definition of $\mathfrak A_{\mathcal E}$, the map $F\colon \IC\to\dom(\mathcal E),\,z\mapsto \Delta_\phi^{iz}b$ is analytic with respect to $\norm\cdot_{\mathcal E}$. Hence
    \begin{equation*}
        \left\lVert\frac 1 h\big(\delta(F(z+h)-\delta(F(z))\big)-\overline\delta(F^\prime(z))\right\rVert_{\mathcal H}^2=\mathcal E\left(\frac 1 h(F(z+h)-F(z))-F^\prime(z)\right)\to 0
    \end{equation*}
    as $h\to 0$. Therefore, $z\mapsto \delta(\Delta_\phi^{iz}b)$ is entire analytic.
\end{proof}

\begin{lemma}\label{lem:core_V_z}
    The set $\operatorname{span}\{\pi_l(a)\cdot\delta(b)\mid a,b\in\mathfrak A_{\mathcal E}\}$ is a core for $\mathcal V_z$ for all $z\in\IC$.
\end{lemma}
\begin{proof}
    Let $D=\operatorname{span}\{\pi_l(a)\cdot\delta(b)\mid a,b\in\mathfrak A_{\mathcal E}\}$. Since $\pi_l(\mathfrak A_{\mathcal E})$ is norm dense in $A_{\mathcal E}$ by definition, the set $D$ is dense in $\mathcal H$. By \Cref{lem:anatytic_ext_V}, it is contained in $\dom(\mathcal V_z)$ and invariant under $(\mathcal V_s)_{s\in\IR}$. It follows from Lemma \ref{lem:core_analytic_generator} that $D$ is a core for $\mathcal V_z$.
\end{proof}

\begin{proposition}
    The operator $\mathcal V_{-i/2}$ is accretive.
\end{proposition}
\begin{proof}
    We have to show that $\Re\langle \xi,\mathcal V_{-i/2}\xi\rangle_{\mathcal H}\geq 0$ for all $\xi\in\dom(\mathcal V_{-i/2})$. By \Cref{lem:core_V_z}, we may assume that $\xi=\sum_{j=1}^m \pi_l(a_j)\cdot\delta(b_j)$ with $a_j,b_j\in\mathfrak A_{\mathcal E}$ for $j\in \{1,\dots,m\}$. Moreover, \Cref{lem:anatytic_ext_V} that implies $\mathcal V_{-i/2}\xi=\sum_{j=1}^m \pi_l(\Delta_\phi^{1/2}a_j)\delta(\Delta_\phi^{1/2}b_j)$. If we combine this with the definition of $\delta$ and $\mathcal H$, we obtain
    \begin{align*}
        \Re\langle\xi,\mathcal V_{-i/2}\xi\rangle&=\Re\sum_{j,k=1}^m \langle \pi_l(a_j)\delta(b_j),\pi_l(\Delta_\phi^{1/2}a_k)\delta(\Delta_\phi^{1/2}b_k)\rangle_{\mathcal H}\\
        &=\lim_{n\to\omega}\Re\sum_{j,k=1}^m \langle \pi_l(a_j)\delta^{(n)}(b_j),\pi_l(\Delta_\phi^{1/2}a_k)\delta^{(n)}(\Delta_\phi^{1/2}b_k)\rangle_{\mathcal H^{(n)}}\\
        &=\lim_{n\to\omega}\Re\left\langle\sum_{j=1}^m \pi_l(a_j)\delta^{(n)}(b_j),\mathcal V^{(n)}_{-i/2}\left(\sum_{j=1}^m \pi_l(a_j)\cdot \delta^{(n)}(b_j)\right)\right\rangle_{\mathcal H^{(n)}}.
    \end{align*}
    This last expression is non-negative by our result in the bounded case.
\end{proof}

After this discussion, we arrive at the following description of (possibly unbounded) conservative completely Dirichlet forms in terms of derivations. Here we say that a Hilbert space $\cH$ is a \emph{Hilbert bimodule} over the unital $C^\ast$-algebra $A$ if it is endowed with commuting unital $\ast$-homomorphisms $\pi_l^\cH\colon A\to B(\cH)$ and $\pi_r^\cH\colon A^{\mathrm{op}}\to B(H)$. As in the case of correspondences, we write $x\xi y$ for $\pi_l^\cH(x)\pi_r^\cH(y^{\mathrm{op}})(y)\xi$ if $x,y\in A$ and $\xi\in\cH$.

\begin{theorem}\label{thm:char_unbdd}
    If $\mathcal E$ is a quantum Dirichlet form, then there exist
    \begin{itemize}
        \item a Hilbert bimodule $\cH$ over $A_{\mathcal E}$,
        \item an anti-unitary involution $\mathcal J\colon\cH\to\cH$,
        \item a strongly continuous one-parameter group $(\cV_s)_{s\in\IR}$ on $\mathcal H$ with $\cV_{-i/2}$ accretive,
        \item a closable operator $\delta\colon \mathfrak A_{\mathcal E}\to\mathcal H$
    \end{itemize}
    satisfying
    \begin{enumerate}[(i)]
        \item $\cH=\overline{\operatorname{span}}\{x\delta(a)y\mid a\in \mathfrak A_{\mathcal E},\,x,y\in A_{\mathcal E}\}$,
        \item $\mathcal J(x\eta y)=y^\ast (\mathcal J\eta)x^\ast$, $\cV_s(x\eta y)=\sigma^\phi_s(x)(\cV_s\eta)\sigma^\phi_s(y)$ for all $x,y\in A_{\mathcal E}$, $\eta\in\cH$, $s\in\IR$,
        \item $\cV_s\mathcal J=\mathcal J \cV_s$ for all $s\in\IR$,
        \item $\delta\circ J=\mathcal J\circ\delta$, $\delta\circ \Delta_\phi^{is}=\cV_s\circ\delta$ for all $s\in\IR$,
        \item $\delta(ab)=\pi_l(a)\delta(b)+\delta(a)\cdot J \pi_r(b)^\ast J$ for all $a,b\in\mathfrak A_{\mathcal E}$
    \end{enumerate}
    such that
    \begin{equation*}
        \dom(\mathcal E)=\dom(\overline\delta),\,\mathcal E(a)=\norm{\overline\delta(a)}_{\mathcal H}^2,\quad a\in\dom(\mathcal E).
    \end{equation*}
\end{theorem}

\section{Connection to Tomita correspondences}\label{sec:Tomita_corr}

In this section we discuss the connection of the representation of quantum Dirichlet forms to Tomita correspondences introduced in \cite{Wir22b}. In \Cref{thm:char_bdd} we have seen that every bounded quantum Dirichlet form $\mathcal E$ gives rise to a self-dual $M$-$M$ correspondence $(\cH,\cJ)$ and a strongly continuous group $(\cV_s)_{s\in \IR}$ satisfying $\cV_s\cJ=\cJ\cV_s$ for all $s\in \IR$ and $\cV_s(x\eta y)=\sigma^\phi_s(x)(\cV_s\eta)\sigma^\phi_s(y)$ for all $x,y\in M$, $\eta\in\cH$ and $s\in\IR$.

However, the group $(\cV_s)_{s\in\IR}$ consists of unitary operators if and only if the quantum Dirichlet form $\mathcal E$ is invariant under $(\Delta_\phi^{is})_{s\in \IR}$ so that $(\cH,\cJ,(\cV_s)_{s\in\IR})$ fails to be a Tomita correspondence as soon as $\mathcal E$ is not invariant under the modular group. In particular, the techniques from \cite{Shl99,Wir22b} to realize $\cH\subset L^2(\hat M)$ as bimodules with $\cJ\subset J_{\hat M}$, where $\hat M$ is a von Neumann algebra that contains $M$ with expectation, do not carry over to the general KMS-symmetric case.

In this section, we show how even if $(\cV_s)_{s\in\IR}$ is not a unitary group, it can be deformed to a strongly continuous unitary group $(\mathcal U_s)_{s\in\IR}$ in such a way that $(\cH,\cJ,(\mathcal U_s)_{s\in\IR})$ becomes a Tomita correspondence.

\begin{lemma}
    Let $(\cH,\cJ)$ be a self-dual $M$-$M$ correspondence and $(\cV_s)_{s\in\IR}$ a strongly continuous group on $\cH$ such that $\cV_s\cJ=\cJ\cV_s$ for all $s\in \IR$ and $\cV_s(x\eta y)=\sigma^\phi_s(x)(\cV_s\eta)\sigma^\phi_s(y)$ for all $x,y\in M$, $\eta\in\cH$ and $s\in\IR$. For $s\in\IR$, if $\cV_s=\tilde{\mathcal U_s}\abs{\cV_s}$ is the polar decomposition of $\cV_s$, then $\abs{\cV_s}$ is a bimodule map, $\tilde{\mathcal U_s}$ is unitary and
    \begin{equation*}
        \tilde{\mathcal U_s}(x\eta y)=\sigma_s(x)(\tilde{\mathcal U}_s\eta)\sigma_s(y)
    \end{equation*}
    for all $x,y\in A_{\mathcal E}$, $\eta\in\cH$.
\end{lemma}
\begin{proof}
    By definition of a strongly continuous group, $\mathcal V_s$ is invertible. Thus the partial isometry $\tilde{\mathcal U}_s$ in the polar decomposition is unitary. Note that
    \begin{align*}
    \mathcal V_s^\ast\mathcal V_s\pi_l^\cH(x)\pi_r^\cH(y)&
    =\mathcal V_s^\ast \pi_l^\cH(\sigma_s(x))\pi_r^\cH(\sigma^\phi_s(y))\mathcal V_s\\
    &=(\pi_l^\cH(\sigma_s(x^\ast))\pi_r^\cH(\sigma_s(y^\ast)) \mathcal V_s)^\ast\mathcal V_s\\
    &=(\mathcal V_s \pi_l^{\mathcal H}(x^\ast)\pi_r^\cH(y^\ast))^\ast\mathcal V_s\\
    &=\pi_l^\cH(x)\pi_r^\cH(y)\mathcal V_s^\ast\mathcal V_s
    \end{align*}
    for all $x,y\in M$. Therefore $\abs{\mathcal V_s}$ is a bimodule map and
\begin{equation*}
    \tilde{\mathcal U}_s(x\eta y)=\cV_s\abs{\cV_s}^{-1}(x\eta y)=\cV_s(x(\abs{\cV_s}^{-1}\eta)y)=\sigma_s(x)(\tilde{\mathcal U}_s\xi)\sigma_s(y)
\end{equation*}
for all $x,y\in A_{\mathcal E}$ and $\eta\in\cH$.
\end{proof}

Usually, $(\tilde{\mathcal U_s})_{s\in\mathbb R}$ is not a group. This can be mended as follows. 

\begin{proposition}\label{prop:deform_unitary_group}
    Let $\mu$ be the measure on $\IR$ with density $\frac{d\mu}{dt}=\frac{2}{\cosh 2\pi t}$, let $(\tilde\cH,\tilde \cJ)$ be a self-dual $M$-$M$ correspondence and let $\cH=L^2(\IR,\mu;\tilde\cH)$. Define the structure of an $M$-$M$ correspondence on $\cH$ by $(x\eta y)(t)=\sigma^\phi_t(x)\eta(t)\sigma^\phi_t(y)$ and an anti-unitary involution $\cJ\colon \cH\to \cH$ by $(\cJ\eta)(t)=\tilde\cJ\eta(t)$. For $s\in\IR$ let
    \begin{equation*}
        \mathcal V_s\colon \cH\to \cH,\,(V_s \xi)(t)=\xi(s+t).
    \end{equation*}
    If $\cV_s=\tilde{\mathcal U}_s\abs{\cV_s}$ is the polar decomposition of $\cV_s$ for $s\in\IR$, then there exists a strongly continuous unitary group $(\mathcal U_s)_{s\in\IR}$ on $\cH$ such that $\tilde{\mathcal U}_{s/n}^n\to \mathcal U_s$ as $n\to\infty$ in strong operator toplogy for all $s\in \IR$. The difference of the generators of $(\cV_s)_{s\in\IR}$ and $(\tilde{\mathcal U}_s)_{s\in\IR}$ is a bounded bimodule map. Moreover, $(\cH,\cJ,(\mathcal U_s)_{s\in\IR})$ is a Tomita correspondence. 
\end{proposition}
\begin{proof}
Let us first compute $\mathcal V_s^\ast$ explicitly. For this purpose, we write $V(t)=-\log \frac 2{\cosh 2\pi t}$. For $\xi,\eta\in L^2(\IR,\mu;\tilde\cH)$ we have
\begin{align*}
    \langle \mathcal V_s\xi,\mathcal V_s\eta\rangle
    &=\int_{\IR} \langle \xi(s+t),\eta(s+t)\rangle e^{-V(t)}\,dt\\
    &=\int_\IR \langle \xi(t),\eta(t)\rangle e^{V(t)-V(t-s)}e^{-V(t)}\,dt.
\end{align*}
Thus $\mathcal V_s^\ast \mathcal V_s \xi(t)=e^{V(t)-V(t-s)}\xi(t)$ and $\tilde{\mathcal U}_s\xi(t)=(\mathcal V_s^\ast\mathcal V_s)^{-1/2}\xi(s+t)=e^{-\frac 1 2(V(s+t)-V(t))}\xi(s+t)$.

We note that
\begin{align*}
    \frac 1 s(\mathcal V_s\xi-\tilde{\mathcal U}_s\xi)(t)=\frac 1 s(1-e^{-\frac 1 2(V(s+t)-V(t))})\xi(s+t)=\frac 1 s(1-e^{-\frac 1 2(V(s+t)-V(t))})\mathcal V_s \xi(t).
\end{align*}
For $x\in [-1,1]\setminus \{0\}$ and $y\in [-1,1]$ let 
\begin{equation*}
    F(x,y)=\frac 1 x\left(1-\left(\frac{1+y}2 e^x+\frac {1-y}2 e^{-x}\right)^{-1/2}\right)
\end{equation*}
and note that $2\pi F(2\pi s,\tanh 2\pi t)=\frac 1 s(1-e^{-\frac 1 2(V(s+t)-V(t))})$.

We have $\frac{1+y}2 e^x+\frac {1-y}2 e^{-x}=1+xy+O(x^2)$ as $x\to 0$ uniformly in $y\in [-1,1]$ and $\left(\frac{1+y}2 e^x+\frac {1-y}2 e^{-x}\right)^{-1/2}=1-\frac{xy}2+O(x^2)$ as $x\to 0$ uniformly in $y\in [-1,1]$, thus $F(x_n,y_n)\to \frac{y}2$ as $x_n\to 0$ and $y_n\to y$. Therefore, $F$ extends to a continuous function on $[-1,1]^2$. In particular, it is bounded.

Hence, if $\xi\in L^2(\mathbb R,\mu;\tilde\cH)$, then
\begin{align*}
    \int_\IR\left\lVert \frac 1 s(1-e^{-\frac 1 2(V(s+t)-V(t))})\xi(t)-\pi\tanh (2\pi t)\xi(t)\right\rVert^2\,d\mu(t)\to 0
\end{align*}
as $s\to 0$ by the dominated convergence theorem.

Together with the strong convergence $\mathcal V_s\to \id$ we get
\begin{equation*}
    \frac 1 s(\mathcal V_s\xi-\tilde{\mathcal U}_s\xi)\to B \xi,
    \end{equation*}
where $B\xi(t)=\frac 1 2 V^\prime(t)\xi(t)$. As $V^\prime\in L^\infty(\mathbb R)$, this implies that whenever $\xi$ belongs to the domain of the generator $A$ of $(\mathcal V_s)$, then
\begin{equation*}
    \frac 1 s(\tilde{\mathcal U}_s\xi-\xi)=\frac 1 s(\tilde{\mathcal U}_s \xi-\mathcal V_s\xi)+\frac  1s(\mathcal V_s\xi-\xi)\to \frac 1 2 V'\xi+A\xi.
\end{equation*}
Since $\pm (A+B)$ is a bounded perturbation of the semigroup generator $\pm A$, it follows that $A+B$ generates a strongly continuous group $(\mathcal U_s)$ on $L^2(\IR,\mu;\tilde\cH)$ and by Chernoff's theorem \cite[Theorem III.5.2]{EN00}, we have
\begin{equation*}
    \mathcal U_s\xi=\lim_{n\to\infty}\tilde{\mathcal U}_{s/n}^n\xi
\end{equation*}
for all $\xi\in L^2(\IR,\mu;\tilde\cH)$ and $s\geq 0$. Moreover, since $-(A+B)$ also generates a semigroup, we see that $\mathcal U_s$ is invertible for every $s\geq 0$ and we can extend $(\mathcal U_s)$ to a strongly continuous group by defining $\mathcal U_{-s}=\mathcal U_s^{-1}$ for $s\geq 0$. Moreover, as a strong limit of unitary operators, $\mathcal U_s$ is an isometry for every $s\in\mathbb R$.

The groups $(\mathcal U_s)$ and $(\mathcal V_s)$ are related by the following identity \cite[Corollary III.1.7]{EN00}:
\begin{equation*}
    \mathcal U_s\xi=\mathcal V_s\xi+\int_0^s \mathcal V_{s-r}B\mathcal U_r\xi\,dr.
\end{equation*}
It follows that
\begin{align*}
    \norm{\mathcal U_s\xi-\mathcal V_s\xi}\leq \abs{s}\pi e^{\pi\abs{s}}\norm{\xi}.
\end{align*}

From $\tilde{\mathcal U}_s(x\xi y)=\sigma_s(x)(\tilde{\mathcal U}_s\xi)\sigma_s(y)$ it is easy to deduce $\mathcal U_s(x\xi y)=\sigma_s(x)(\mathcal U_s \xi)\sigma_s(y)$. Furthermore, since $\mathcal V_s$ commutes with $\mathcal J$ and $\mathcal J$ is self-adjoint, $\mathcal U_s$ also commutes with $\mathcal J$. Therefore, $(L^2(\IR,\mu;\tilde\cH),\mathcal J,(\mathcal U_s)_{s\in\IR})$ is a Tomita correspondence.
\end{proof}

\begin{theorem}\label{thm:Tomita_corr_Dir_form}
    If $\mathcal E$ is a bounded quantum Dirichlet form, then there exists a von Neumann algebra $\hat M\supset M$ with a faithful normal conditional expectation $E\colon \hat M\to M$, a bounded self-adjoint $M$-bimodule map $B\colon L^2(\hat M)\to L^2(\hat M)$ and a bounded linear map $\delta\colon L^2(M)\to L^2(\hat M)$ satisfying
    \begin{itemize}
        \item $J_{\phi\circ E}\circ\delta=\delta\circ J_\phi$
        \item $e^{s(i\log\Delta_{\phi\circ E}+B)}\circ \delta=\delta\circ \Delta_\phi^{is}$, where $(e^{s(i\log\Delta_{\phi\circ E}+B)})_{s\in\IR}$ is the group generated by $i\log\Delta_{\phi\circ E}-B$,
        \item $\delta(ab)=\pi_l(a)\delta(b)+\delta(a)\cdot J_{\hat M}\pi_r(b)^\ast J_{\hat M}$ for all $a,b\in\mathfrak A_{\mathcal E}$
        \end{itemize}
    such that
    \begin{equation*}
        \mathcal E(a)=\norm{\delta(a)}_{L^2(\hat M)}^2,\quad a\in L^2(M).
    \end{equation*}
\end{theorem}
\begin{proof}
    By \Cref{prop:deform_unitary_group}, we get a Tomita correspondence $(\mathcal H,\mathcal J,(\mathcal U_s))$, a strongly continuous group $(\mathcal V_s)$ on $\mathcal H$ and a bounded derivation $\delta\colon L^2(M)\to \mathcal H$ such that $\delta\circ J_\phi=\mathcal J\circ \delta$, $\delta\circ \Delta_\phi^{is}=\mathcal V_s\circ \delta$ and $\mathcal E(a)=\norm{\delta(a)}_2^2$. Moreover, if $A$ denotes the generator of $\mathcal V_s$, then there exists a bounded self-adjoint bimodule map $B\colon \mathcal H\to \mathcal H$ such that $A+B$ is the generator of $(\mathcal U_s)$.

    Following \cite[Section 4.2]{Wir22b}, we can define the operator-valued free Araki--Woods algebra $\hat M=\Gamma_0(\mathcal H,\mathcal J,(\mathcal U_s))$ over $(\cH,\cJ,(\mathcal U_s)_{s\in \IR})$, which comes with a canonical faithful normal conditional expectation $E\colon \hat M\to M$ and a canonical identification $L^2(\hat M)\cong \mathcal F(_M\mathcal H_M)\supset \mathcal H$ \cite[Theorem 4.15]{Wir22b}. Moreover, $\Delta_{\phi\circ E}^{is}|_\mathcal H=\mathcal U_s$ by \cite[Theorem 5.18]{KW25}.

    Since $B\colon\mathcal H\to \mathcal H$ is a bounded self-adjoint bimodule map, we can extend it to a bounded self-adjoint bimodule map on $\mathcal F(_M\mathcal H_M)$ by defining $B|_{L^2(M)}=\id_{L^2(M)}$ and $B|_{\mathcal H^{\otimes_\phi n}}=\norm{B}(B/\norm{B})^{\otimes_\phi n}$. Since both $B$ and $\Delta_\phi$ leave $\cH$ invariant, it follows that the group generated by $i\log\Delta_{\phi\circ E}-B$ leaves $\cH$ invariant, and since its generator is an extension of the generator of $(\mathcal V_s)_{s\in\IR}$, it is easy to see that $\mathcal V_s$ is a restriction of $e^{s(i\log\Delta_{\phi\circ E}+B)}$
\end{proof}

\bibliographystyle{alpha}
\bibliography{ref}

\end{document}